\documentclass[times]{oupau}
\usepackage{import}

\usepackage{amsthm}
\usepackage{amssymb}
\usepackage{stmaryrd}
\usepackage{wasysym}
\usepackage{mathrsfs}
\usepackage{hyperref}
\usepackage{enumitem}
\usepackage{dsfont}
\usepackage{soul}
\usepackage{filecontents}
\hypersetup{colorlinks=true,linkcolor=black,citecolor=black}
\usepackage{color}
\usepackage[all]{xy}
\usepackage{float}
\usepackage{bm}
\usepackage{mathtools}
\usepackage{thmtools}
\usepackage{setspace}
\usepackage{comment}
\usepackage{tikz}
\usepackage{mathdots}
\usepackage{cleveref}

\bibpunct{[}{]}{,}{n}{,}{;}

\numberwithin{equation}{section}
\newtheorem{thm}{Theorem}[section]
\newtheorem{prop}[thm]{Proposition}

\newtheorem{lem}[thm]{Lemma}
\newtheorem{cor}[thm]{Corollary}

\theoremstyle{definition}
\newtheorem{defi}{Definition}

\newtheorem{rmk}{Remark}

\newcommand\reallywidehat[1]{%
\savestack{\tmpbox}{\stretchto{%
  \scaleto{%
    \scalerel*[\widthof{\ensuremath{#1}}]{\kern-.6pt\bigwedge\kern-.6pt}%
    {\rule[-\textheight/2]{1ex}{\textheight}}
  }{\textheight}%
}{0.5ex}}%
\stackon[1pt]{#1}{\tmpbox}%
}
\DeclareSymbolFont{bbold}{U}{bbold}{m}{n}
\DeclareSymbolFontAlphabet{\mathbbold}{bbold}

\newcommand{\R}{\mathbb{R}}
\newcommand{\Z}{\mathbb{Z}}
\newcommand{\Q}{\mathbb{Q}}
\newcommand{\N}{\mathbb{N}}
\newcommand{\C}{\mathbb{C}}
\newcommand{\F}{\mathbb{F}}

\newcommand{\bbM}{\mathbb{M}}

\newcommand{\bbone}{\mathbbold{1}}

\renewcommand{\Im}{\text{Im}}

\newcommand{\inj}{\hookrightarrow}
\newcommand{\tth}{^{th}}

\newcommand{\M}{\mu_{Haar}}

\newcommand{\tkappa}{\tilde{\kappa}}

\renewcommand{\L}{\Lambda}
\newcommand{\Y}{\mathbb{Y}}

\newcommand{\G}{\mathscr{G}_t}

\newcommand{\GG}{\widetilde{\mathscr{G}_t}}

\newcommand{\ba}{\mathbf{a}}

\newcommand{\bx}{\mathbf{x}}
\newcommand{\by}{\mathbf{y}}

\newcommand{\cP}{\mathbb{Y}}

\newcommand{\tL}{\tilde{L}}

\newcommand{\tE}{\tilde{E}}
\newcommand{\EM}{E}
\newcommand{\cM}{\mathcal{M}}

\DeclareMathOperator{\Sym}{Sym}
\DeclareMathOperator{\len}{len}

\DeclareMathOperator{\Id}{Id}

\DeclareMathOperator{\Aut}{Aut}

\DeclareMathOperator{\GL}{GL}

\DeclareMathOperator{\coker}{coker}
\DeclareMathOperator{\Sig}{\mathbb{GT}}
\DeclareMathOperator{\bSig}{\overline{\mathbb{GT}}}

\DeclareMathOperator{\SN}{SN}
\DeclareMathOperator{\ESN}{ESN}
\DeclareMathOperator{\diag}{diag}
\DeclareMathOperator{\Supp}{Supp}
\DeclareMathOperator{\Mat}{Mat}

\newcommand{\sqbinom}[2]{\begin{bmatrix}#1\\ #2\end{bmatrix}}
\newcommand{\qhyp}{\;_4 \bar{\phi}_3}
\newcommand{\qhypl}{\;_2 \bar{\phi}_1}
\newcommand{\qhypm}{\;_3 \bar{\phi}_2}

\begin{document}

\title{Hall-Littlewood polynomials, boundaries, and $p$-adic random matrices}
\shorttitle{Hall-Littlewood polynomials, boundaries, and $p$-adic random matrices}

\author{Roger Van Peski\affil{1}}
\abbrevauthor{R. Van Peski}
\headabbrevauthor{Van Peski, R.}
\address{\affilnum{1}Department of Mathematics, Massachusetts Institute of Technology, USA}

\correspdetails{rvp@mit.edu}

\received{1 Month 20XX}
\revised{11 Month 20XX}
\accepted{21 Month 20XX}

\communicated{A. Editor}



\begin{abstract} 
We prove that the boundary of the Hall-Littlewood $t$-deformation of the Gelfand-Tsetlin graph is parametrized by infinite integer signatures, extending results of Gorin \cite{gorin2012q} and Cuenca \cite{cuenca2018asymptotic} on boundaries of related deformed Gelfand-Tsetlin graphs. In the special case when $1/t$ is a prime $p$ we use this to recover results of Bufetov-Qiu \cite{bufetov2017ergodic} and Assiotis \cite{assiotis2020infinite} on infinite $p$-adic random matrices, placing them in the general context of branching graphs derived from symmetric functions.

Our methods rely on explicit formulas for certain skew Hall-Littlewood polynomials. As a separate corollary to these, we obtain a simple expression for the joint distribution of the cokernels of products $A_1, A_2A_1, A_3A_2A_1,\ldots$ of independent Haar-distributed matrices $A_i$ over $\Z_p$, generalizing the explicit formula for the classical Cohen-Lenstra measure.
\end{abstract}
\maketitle
\tableofcontents
\section{Introduction}

\subsection{Hall-Littlewood polynomials.}

The classical \emph{Hall-Littlewood polynomials} $P_\lambda(x_1,\ldots,x_n; t)$ are a family of symmetric polynomials in variables $x_1,\ldots,x_n$, with an additional parameter $t$, indexed by weakly decreasing sequences of nonnegative integers $\lambda = (\lambda_1 \geq \lambda_2 \geq \ldots \geq \lambda_n)$ (called nonnegative signatures). They reduce to Schur polynomials at $t=0$ and monomial symmetric polynomials at $t=1$, and play key roles in geometry, representation theory, and algebraic combinatorics. For this work, the most relevant role is that for $t=1/p$, $p$ prime, they are intimately related to $\GL_n(\Z_p)$-spherical functions on $\GL_n(\Q_p)$ \cite[Chapter V]{mac}, and consequently are important in $p$-adic random matrix theory \cite{van2020limits}. 

Explicitly they are defined by 
\begin{equation}\label{eq:hl_explicit_intro}
P_\lambda(x_1,\ldots,x_n;t) :=  \frac{1}{v_\lambda(t)} \sum_{\sigma \in S_n} \sigma\left(x_1^{\lambda_1} \cdots x_n^{\lambda_n} \prod_{1 \leq i < j \leq n} \frac{x_i-tx_j}{x_i-x_j}\right),
\end{equation}
where $\sigma$ acts by permuting the variables and $v_\lambda(t)$ is the normalizing constant such that the $x_1^{\lambda_1} \cdots x_n^{\lambda_n}$ term has coefficient $1$.  As with other families of symmetric functions one may define \emph{skew Hall-Littlewood polynomials} $P_{\lambda/\mu}$ in terms of two nonnegative signatures $\lambda,\mu$ of lengths $n,k$ by 
\begin{equation}\label{eq:skew_def_intro}
    P_\lambda(x_1,\ldots,x_n;t) = \sum_{\mu \in \Sig_k} P_{\lambda/\mu}(x_1,\ldots,x_{n-k};t) P_\mu(x_{n-k+1},\ldots,x_n;t),
\end{equation}
where $\Sig_k$ is the set of signatures of length $k$.

\subsection{Branching graphs from Hall-Littlewood polynomials.}

In 1976, Voiculescu \cite{voiculescu1976representations} classified the characters of the infinite unitary group $U(\infty)$, defined as the inductive limit of the chain $U(1) \subset U(2) \subset \ldots$. This was later shown to be equivalent to earlier results by Aissen, Edrei, Schoenberg and Whitney, stated without reference to representation theory. A similar story unfolded for the infinite symmetric group $S_\infty$ \cite{kerov1998boundary,vershik1981asymptotic}, related to the classical Thoma theorem \cite{thoma1964unzerlegbaren}. See \cite[\S 1.1]{borodin2012boundary} and the references therein for a more detailed exposition of both.

In later works such as \cite{vershik1982characters,okounkov1998asymptotics} the result for $U(\infty)$ was recast in terms of classifying the boundary of the so-called \emph{Gelfand-Tsetlin branching graph}, defined combinatorially in terms of Schur polynomials. This led to natural generalizations to other branching graphs defined in terms of degenerations of Macdonald polynomials $P_\lambda(x_1,\ldots,x_n;q,t)$, which feature two parameters $q,t$ and specialize to Hall-Littlewood polynomials when $q=0$; see \cite{borodin2012boundary,cuenca2018asymptotic,gorin2012q,okounkov1998asymptotics,olshanski2021macdonald}. In special cases these combinatorial results take on additional significance in representation theory and harmonic analysis; the Schur case was already mentioned, and two other special cases of the result of \cite{okounkov1998asymptotics} for the Jack polynomial case specialize to statements about the infinite symmetric spaces $U(\infty)/O(\infty)$ and $U(2\infty)/Sp(\infty)$. For the Young graph, the boundary of its natural Hall-Littlewood deformation was conjectured in equivalent form in \cite{kerov1992generalized}, proved in \cite{matveev2019macdonald}, and used to deduce results on infinite matrices over finite fields in \cite{cuenca2022infinite}. Surprisingly, however, the boundary of the Hall-Littlewood deformation of the Gelfand-Tsetlin graph has not previously been carried out, despite the fact that the appearance of Hall-Littlewood polynomials in harmonic analysis on $p$-adic groups suggests interpretations beyond the purely combinatorial setting. 

Let us describe the setup of the Hall-Littlewood branching graph; we refer to \cite[Chapter 7]{borodin2017representations} for an expository account of the general formalism of graded graphs and their boundaries. Let $\Sig_n = \{(\lambda_1,\ldots,\lambda_n) \in \Z^n: \lambda_1 \geq \ldots \geq \lambda_n\}$ be the set of \emph{integer signatures of length $n$}, not necessarily nonnegative. Allowing $\lambda$ to be an arbitrary signature, \eqref{eq:hl_explicit_intro} yields a symmetric `Hall-Littlewood Laurent polynomial' which we also denote $P_\lambda$. Let $\G$ be the weighted graph with vertices 
\[
\bigsqcup_{n \geq 1} \Sig_n
\]
and edges between $\lambda \in \Sig_n, \mu \in \Sig_{n+1}$ with weights 
\begin{equation}\label{eq:links_intro}
    L_n^{n+1}(\mu,\lambda) := P_{\mu/\lambda}(t^n;t) \frac{P_\lambda(1,\ldots,t^{n-1};t)}{P_\mu(1,\ldots,t^n;t)},
\end{equation}
known as \emph{cotransition probabilities}. These cotransition probabilities are stochastic by \eqref{eq:skew_def_intro},  so any probability measure on $\Sig_{n+1}$ induces another probability measure on $\Sig_n$. A sequence of probability measures $(M_n)_{n \geq 1}$ which is consistent under these maps is called a \emph{coherent system}. As these form a simplex, understanding coherent systems reduces to understanding the extreme points, called the \emph{boundary} of the branching graph. Our first main result is an explicit description of the boundary of $\G$. Here $\nu_x' = \#\{i: \nu_i \geq x\}$, $(a;t)_n = \prod_{i=1}^n (1-at^{i-1})$ is the $t$-Pochhammer symbol, 
\[
\sqbinom{a}{b}_t = \frac{(t;t)_a}{(t;t)_b (t;t)_{a-b}}
\]
is the $t$-binomial coefficient, and we let $\Sig_\infty$ be the set of weakly decreasing tuples of integers $(\mu_1,\mu_2,\ldots)$.

\begin{restatable}{thm}{boundary}\label{thm:boundary}
For any $t \in  (0,1)$, the boundary of $\G$ is naturally in bijection with $\Sig_\infty$. Under this bijection $\mu \in \Sig_\infty$ corresponds to the coherent system $(M^\mu_n)_{n \geq 1}$ defined explicitly by
\begin{equation*}
    M_n^\mu(\lambda) := 
    (t;t)_n \prod_{x \in \Z} t^{(\mu'_x - \lambda'_x)(n-\lambda'_x)} \sqbinom{\mu_x' - \lambda_{x+1}'}{\lambda_x' - \lambda_{x+1}'}_t 
\end{equation*}
for $\lambda \in \Sig_n$. 
\end{restatable}

We note that the product over $x \in \Z$ in fact has only finitely many nontrivial terms. The fact that the extreme measures have simple explicit formulas is unusual for results of this type--usually, the measures are characterized implicitly by certain generating functions.

The proof in \Cref{sec:branching_graphs} follows the general outline of the so-called \emph{Vershik-Kerov ergodic method}, as do those of many related results mentioned above. One of the closest works to our setting is \cite{gorin2012q}, which studies the Schur analogue with edge weights
\[
s_{\mu/\lambda}(t^n) \frac{s_\lambda(1,\ldots,t^{n-1})}{s_\mu(1,\ldots,t^n)}
\]
for $t \in (0,1)$, where $s_\lambda(x)$ is the Schur polynomial. The boundary is shown to be naturally in bijection with $\Sig_\infty$ as in our case\footnote{Our $t$ corresponds to the $q^{-1}$ in the notation \cite{gorin2012q}. The setting of \cite{gorin2012q} actually corresponds to $t>1$, and the boundary corresponds to infinite \emph{increasing} tuples of integers, but this statement is equivalent to ours upon interchanging signatures with their negatives--see the comment after Theorem 1.1 in \cite{gorin2012q}.}. 

The boundary classification results of \cite{gorin2012q} are generalized in \cite{cuenca2018asymptotic} to the Macdonald case with cotransition probabilities
\begin{equation}\label{eq:mac_links}
 P_{\mu/\lambda}(t^n;q,t=q^k) \frac{P_\lambda(1,\ldots,t^{n-1};q,t=q^k)}{P_\mu(1,\ldots,t^n;q,t=q^k)}   
\end{equation}
for any $k \in \N$, and the boundary is again identified with $\Sig_\infty$; when $k=1$ this reduces to the result of \cite{gorin2012q}. We do not see how \Cref{thm:boundary} could be accessed by the methods of \cite{cuenca2018asymptotic} or the newer work \cite{olshanski2021macdonald}, which treats the related \emph{Extended Gelfand-Tsetlin graph} with weights coming from Macdonald polynomials with arbitrary $q,t \in (0,1)$. Instead, we rely on explicit expressions, \Cref{thm:finite_skew_fn_computation} and \Cref{thm:skew_formula}, for the skew Hall-Littlewood polynomials appearing in \eqref{eq:links_intro}. This means that \Cref{thm:boundary} gives explicit formulas for the extreme coherent measures, while in previous works they were defined implicitly by certain generating functions.

\subsection{Ergodic measures on infinite $p$-adic random matrices.} 

In the special case $t=1/p$, the purely combinatorial results on Hall-Littlewood polynomials have consequences in $p$-adic random matrix theory, and we may deduce results of \cite{bufetov2017ergodic,assiotis2020infinite} from \Cref{thm:boundary} above. We refer to \Cref{sec:matrix_products} for basic background on the $p$-adic integers $\Z_p$ and $p$-adic field $\Q_p$. The group $\GL_n(\Z_p) \times \GL_m(\Z_p)$ acts on $\Mat_{n \times m}(\Q_p)$ by left- and right multiplication, and the orbits of this action on nonsingular matrices are parametrized by the set $\bSig_{\min(m,n)}$ of `extended' signatures with parts allowed to be equal to $-\infty$. Explicitly, for any $A \in \Mat_{n \times m}(\Q_p), n \leq m$ there exist $U \in \GL_n(\Z_p), V \in \GL_m(\Z_p)$ such that 
\[
UAV = \diag_{n \times m}(p^{-\lambda_1},\ldots,p^{-\lambda_n})
\]
for some $\lambda \in \bSig_n$, where we take $p^\infty=0$ by convention. The extended signature $\lambda$ is unique, and we refer to the $\lambda_i$ as the \emph{singular numbers} of $A$ and write $\SN(A) = \lambda \in \bSig_n$.

For fixed $n \leq m$, the $\GL_n(\Z_p) \times \GL_m(\Z_p)$ bi-invariant measures on $\Mat_{n \times m}(\Q_p)$ are all convex combinations of those parametrized by $\bSig_n$ via $U \diag_{n \times m}(p^{\lambda_1},\ldots,p^{\lambda_n}) V$ with $U,V$ distributed by the Haar measures on $\GL_n(\Z_p),\GL_m(\Z_p)$ respectively. One may define $\GL_\infty(\Z_p)$ as a direct limit of the system 
\[
\GL_1(\Z_p) \subset \GL_2(\Z_p) \subset \ldots
\]
and it is natural to ask for the extension of this result to infinite matrices, i.e. for the extreme points in the set of $\GL_\infty(\Z_p)$ bi-invariant measures on $\Mat_{\infty \times \infty}(\Q_p)$. This problem was previously solved in \cite{bufetov2017ergodic}, which gave an explicit family of measures in bijection with $\bSig_\infty$. We give a new proof that the extreme measures are naturally parametrized by $\bSig_\infty$ in \Cref{thm:recover_BQ} below.

\begin{restatable}{thm}{recoverBQ}\label{thm:recover_BQ}
The set of extreme $\GL_\infty(\Z_p) \times \GL_\infty(\Z_p)$-invariant measures on $\Mat_{\infty \times \infty}(\Q_p)$ is naturally in bijection with $\bSig_\infty$. Under this bijection, the measure $E_\mu$ corresponding to $\mu \in \bSig_\infty$ is the unique measure such that its $n \times m$ truncations are distributed by the unique $\GL_n(\Z_p) \times \GL_m(\Z_p)$-invariant measure on $\Mat_{n \times m}(\Q_p)$ with singular numbers distributed according to the measure $M^\mu_{m,n}$ defined in \Cref{thm:double_boundary} in the case $t=1/p$.
\end{restatable}

Our proof goes by deducing this parametrization by $\bSig_\infty$ from an augmented version of the parametrization by $\Sig_\infty$ appearing in \Cref{thm:boundary}. The key fact which relates the random matrix setting to the purely combinatorial setting is a result, stated later as \Cref{thm:p-adic_corners}, which was proven originally in \cite{van2020limits}. This result relates the distribution of singular numbers of a $p$-adic matrix after removing a row or column to the cotransition probabilities \eqref{eq:links_intro}. 

We note that while Hall-Littlewood polynomials are not mentioned by name in \cite{bufetov2017ergodic}, it should be possible to extrapolate many of their Fourier analytic methods to statements about Hall-Littlewood polynomials at general $t$. Our methods, which are based on explicit formulas for certain skew Hall-Littlewood polynomials, nonetheless differ substantially from those of \cite{bufetov2017ergodic} in a manner which is not merely linguistic. Let us also be clear that while both \Cref{thm:recover_BQ} and \cite{bufetov2017ergodic} show that the extreme measures are parametrized by $\mu \in \bSig_\infty$, it is not obvious from the descriptions that the measures corresponding to a given $\mu \in \bSig_\infty$ under \cite{bufetov2017ergodic} and \Cref{thm:recover_BQ} are in fact the same. A separate argument, assuming the result of \cite{bufetov2017ergodic}, is required to prove that the two parametrizations by $\bSig_\infty$ match, see \Cref{thm:emu=temu}. This additionally provides a computation of the distribution of singular numbers of finite corners of matrices drawn from the measures in \cite{bufetov2017ergodic}, which is new. We refer to \Cref{rmk:BQ_differences} for more detail on the differences between \Cref{thm:recover_BQ} and \cite[Theorem 1.3]{bufetov2017ergodic}, in particular an explanation of how our results carry over to a general non-Archimedean local field as is done in \cite{bufetov2017ergodic}. We mention also that the other main result of \cite{bufetov2017ergodic} is a classification of the extreme measures on infinite symmetric matrices $\{A \in \Mat_{\infty \times \infty}(\Q_p): A^T = A\}$ invariant under $\GL_\infty(\Z_p)$; it would be interesting to have an analogous Hall-Littlewood proof of this result as well, see \Cref{rmk:symm_HL_proof} for further discussion of possible strategy and difficulties.

\begin{rmk}
In addition to \cite{okounkov1998asymptotics}, another work somewhat similar in spirit to \Cref{thm:boundary} and \Cref{thm:recover_BQ} is \cite{assiotis2021boundary}. This work finds the boundary of a certain branching graph defined via multivariate Bessel functions--another degeneration of Macdonald polynomials--and related to $\beta$-ensembles at general $\beta$. In the classical values $\beta=1,2,4$ this recovers results on branching graphs coming from random matrix theory. Results such as ours in terms of Hall-Littlewood polynomials may be regarded as extrapolations of $p$-adic random matrix theory to arbitrary real $p>1$ in the same way $\beta$-ensembles extrapolate classical random matrix theory to real $\beta > 0$, see also \cite[Remark 13]{van2020limits}. 
\end{rmk}

\subsection{Ergodic decompositions of $p$-adic Hua measures.}

For finite random matrices over $\Q_p$ or $\C$, one wishes to compute the distribution of singular numbers, singular values or eigenvalues of certain distinguished ensembles such as the classical GUE, Wishart and Jacobi ensembles (over $\C$), or the additive Haar measure over $\Z_p$. The infinite-dimensional analogue of this problem is to compute how distinguished measures on infinite matrices decompose into extreme points, which correspond to ergodic measures. Such a decomposition is given by a probability measure on the space of ergodic measures, which in our case corresponds to a probability measure on $\bSig_\infty$.

One such distinguished family of measures on $p$-adic matrices is given by the \emph{$p$-adic Hua measures} defined in \cite{neretin2013hua}, which are analogues of the complex Hua-Pickrell measures\footnote{See \cite{borodin2001infinite}, which coined the term for these measures, for an historical discussion of these measures and summary of the contents of the earlier works \cite{hua1963harmonic,pickrell1987measures}.}. There is a $p$-adic Hua measure $\bbM_n^{(s)}$ on $\Mat_{n \times n}(\Q_p)$ for each $n \in \Z_{\geq 1}, s \in \R_{>-1}$, which is defined by an explicit density with respect to the underlying additive Haar measure on $\Mat_{n \times n}(\Q_p)$, see \Cref{def:p-hua}. A motivating property of these measures is that they are consistent under taking corners, and hence define a measure $\bbM_\infty^{(s)}$ on $\Mat_{\infty \times \infty}(\Q_p)$. The decomposition of this measure into ergodic measures on $\Mat_{\infty \times \infty}(\Q_p)$ was computed recently in \cite{assiotis2020infinite}, and we reprove the result using the aforementioned relation between $p$-adic matrix corners and the Hall-Littlewood branching graph $\G$. Below $E_\mu$ is as in \Cref{thm:recover_BQ}, $\Y$ is the set of integer partitions, $Q_\lambda$ is the dual normalization of the Hall-Littlewood symmetric function, and the normalizing constant $\Pi(1,\ldots;u,\ldots)$ is the so-called \emph{Cauchy kernel}--see \Cref{sec:prelim} for precise definitions.

\begin{restatable}{thm}{recoverA}\label{thm:recover_assiotis}
Fix a prime $p$ and real parameter $s > -1$, and let $t=1/p$ and $u=p^{-1-s}$. Then the infinite $p$-adic Hua measure $\bbM_\infty^{(s)}$ decomposes into ergodic measures according to
\begin{equation}\label{eq:phua_decomp}
    \bbM_\infty^{(s)} = \sum_{\mu \in \Y} \frac{P_\mu(1,t,\ldots;t) Q_\mu(u,ut,\ldots;t)}{\Pi(1,\ldots;u,\ldots)} E_\mu
\end{equation}
where $E_\mu$ is as defined in \Cref{thm:recover_BQ}.
\end{restatable}

The key ingredient in the original proof of \Cref{thm:recover_assiotis} given previously in \cite{assiotis2020infinite} is a certain Markov chain which generates the finite Hua measures $\bbM_n^{(s)}$, and which was guessed from Markov chains appearing in similar settings \cite{fulman_main}. The arguments there did not use Hall-Littlewood polynomials, but the limiting measure on $\bSig_\infty$ which describes the ergodic decomposition was observed in \cite{assiotis2020infinite} to be the so-called Hall-Littlewood measure in \eqref{eq:phua_decomp}, by matching explicit formulas. From our perspective, by contrast, the fact that this measure is a Hall-Littlewood measure is natural and is key to the proof.  

\subsection{From Hall-Littlewood polynomials to cokernels of products of $p$-adic random matrices.}

In another direction, random $p$-adic matrices have been subject to much activity in arithmetic statistics going back to the 1983 conjectures of Cohen and Lenstra \cite{cohen-lenstra} on class groups of quadratic imaginary number fields, and their interpretation via random matrices in \cite{friedman-washington}. These works interpret the singular numbers of a random matrix $A \in \Mat_n(\Z_p)$ as specifying a random abelian $p$-group: if $\SN(A) = -\lambda$ with $\lambda_n \geq 0$, then viewing $A$ as a map $\Z_p^n \to \Z_p^n$ one has
\[
\coker(A) = \Z_p^n/\Im(A) \cong \bigoplus_{i=1}^n \Z/p^{\lambda_i}\Z =: G_\lambda(p).
\]
For $A_n \in \Mat_{n \times n}(\Z_p)$ with iid entries distributed according to the additive Haar measure on $\Z_p$, the result of \cite{friedman-washington} implies
\begin{equation}\label{eq:cl_intro}
    \lim_{n \to \infty} \Pr(\coker(A_n) \cong G_\lambda(p)) = \frac{1}{Z}P_\lambda(1,t,\ldots;t) Q_\lambda(t,t^2,\ldots;t) = \frac{1}{Z} |\Aut(G_\lambda(p))|^{-1}
\end{equation}
where $t=1/p$ and $Z = \Pi(1,t,\ldots;t,t^2,\ldots)$ is a normalizing constant. For odd $p$ this distribution was conjectured to describe the $p$-torsion parts of class groups random quadratic imaginary number fields ordered by discriminant, and is often called the Cohen-Lenstra distribution \cite{cohen-lenstra}.

The next result generalizes the finite-$n$ version of \eqref{eq:cl_intro} to arbitrary products of independent additive Haar matrices. Here $n(\lambda) := \sum_{i=1}^n (i-1)\lambda_i$ for $\lambda \in \Sig_n$.

\begin{thm}\label{thm:product_intro}
Let $t=1/p$, fix $n \geq 1$, and let $A_i$ be iid $n \times n$ matrices with iid entries distributed by the additive Haar measure on $\Z_p$. Then the joint distribution of $\coker(A_1), \coker(A_2A_1), \ldots$ is given by 
\begin{align}\label{eq:product_intro}
\begin{split}
    & \Pr(\coker(A_i \cdots A_1) \cong G_{\lambda(i-1)}(p) \text{ for all }i=1,\ldots,k) \\
     &= (t;t)_n^k t^{n(\lambda(k))} \prod_{1 \leq i \leq k} \prod_{x \in \Z} t^{\binom{\lambda(i)_x' - \lambda(i-1)_x'+1}{2}} \sqbinom{\lambda(i)_x' - \lambda(i-1)_{x+1}'}{\lambda(i)_x' - \lambda(i)_{x+1}'}_t
\end{split}
\end{align}
for any $k$ and $\lambda(1),\ldots,\lambda(k) \in \Sig_n^{\geq 0}$, where we take $\lambda(0) = (0,\ldots,0)$ in \eqref{eq:product_intro}.
\end{thm}

Note that the product over $x \in \Z$, which may appear uninviting, in fact has only finitely many nontrivial terms. As a special case one obtains the prelimit version of \eqref{eq:cl_intro}, due to \cite{friedman-washington}: for $A \in \Mat_n(\Z_p)$ with iid additive Haar entries,
\begin{equation}\label{eq:fw_intro}
\Pr(\coker(A) \cong G_\lambda(p)) = t^{2n(\lambda)+|\lambda|} \frac{(t;t)_n^2}{\prod_{i \geq 0} (t;t)_{m_i(\lambda)}}
\end{equation}
where $|\lambda| = \sum_i \lambda_i$.

At first sight it might be unclear why \Cref{thm:product_intro} is a natural generalization to undertake, but we believe it to be in light of the wealth of other natural $p$-adic random matrix ensembles which have found applications in number theory and combinatorics. The iid Haar measure on nonsquare matrices was used in \cite{wood2015random} to model $p$-torsions of class groups of real quadratic number fields, and the corresponding measure on cokernels was also related to Hall-Littlewood polynomials in \cite{van2020limits}. Measures on symmetric and antisymmetric $A$ have been studied in \cite{wood2017distribution} and \cite{bhargava2013modeling}, respectively, as models of sandpile groups of random graphs and Tate-Shafarevich groups of elliptic curves, and have been related to Hall-Littlewood polynomials in \cite{fulman2016hall} and \cite{fulman2018random} respectively. Given the utility of these models, it seems that any natural enough distribution on $p$-adic random matrices is likely to model some class of random abelian $p$-groups appearing in nature.

An expression for the probability in \eqref{eq:product_intro} was derived in \cite{van2020limits}, but in terms of skew Hall-Littlewood polynomials rather than the explicit formula appearing above. That result was suitable for asymptotics as the number of products went to infinity, but seemed less adapted to studying the kinds of arithmetic questions studied in the literature, such as the probability that the cokernel is a cyclic group, for a finite number of products. The explicit nature of \Cref{thm:product_intro} appears more promising in this regard.

We note that one may give efficient sampling algorithms for the distribution in \Cref{thm:product_intro} by interpreting the RHS of \eqref{eq:product_intro} in terms of steps of a certain Markov chain, generalizing \cite[Theorem 10]{fulman_main}. It should also be possible to interpret the $n \to \infty$ limit of \eqref{eq:product_intro} in terms of the appropriate notion of automorphisms of a nested sequence of abelian $p$-groups, generalizing the $k=1$ case \eqref{eq:cl_intro}; we have not attempted to address this question but hope it will be taken up in the future. 

\subsection{Proof methods and skew Hall-Littlewood formulas.}\label{subsec:skew_hl}

A classical fact which follows from \eqref{eq:hl_explicit_intro} is that when a geometric sequence with common ratio $t$ is substituted in for the variables $x_i$, the so-called \emph{principal specialization}, the Hall-Littlewood polynomial takes a particularly simple form: 
\begin{equation}\label{eq:princ_intro}
    P_\lambda(u,ut,\ldots,ut^{n-1};t) = u^{|\lambda|} t^{n(\lambda)} \frac{(t;t)_n}{\prod_{x \in \Z} (t;t)_{m_x(\lambda)}}.
\end{equation}
An explicit formula such as \eqref{eq:hl_explicit_intro} is lacking for the skew Hall-Littlewood polynomials, but their principal specializations still have a relatively simple form. In the case of a specialization $u,ut,\ldots$ in infinitely many variables this reads 
\begin{equation}\label{eq:inf_skew_princ_intro}
    P_{\mu/\lambda}(u,ut,\ldots; t) = (t;t)_\infty u^{|\mu|-|\lambda|} t^{n(\mu/\lambda)}\prod_{x > 0} \frac{(t^{1+\mu_x'-\lambda_x'};t)_{m_x(\lambda)}}{(t;t)_{m_x(\mu)}}
\end{equation}
for integer partitions $\mu,\lambda$, where $n(\mu/\lambda)$ generalizes $n(\lambda)$ to skew diagrams--see \Cref{def:gen:n(lambda)}. The formula \eqref{eq:inf_skew_princ_intro} above is equivalent to a formula for the modified Hall-Littlewood polynomials \cite[Theorem 3.1]{kirillov1998new}, see also \cite{garbali2020modified,warnaar2013remarks}. We give a different proof in \Cref{sec:skew_formulas} by degenerating formulas for principally specialized skew higher spin Hall-Littlewood polynomials, recently shown in \cite{borodin2018higher_original_paper}, partially because we additionally need the result when the geometric progression is finite, see \Cref{thm:finite_skew_fn_computation}. These formulas are the key technical input in the proof of \Cref{thm:boundary}, and the proof of that result in \Cref{sec:branching_graphs} relies on using these formulas to prove certain estimates (which are not extremely difficult, once one has the formulas). They also imply nontrivial variants of the \emph{skew Cauchy identity} in the special case of principal specializations, which are derived and used in the algebraic manipulations in the proof of \Cref{thm:recover_assiotis}. In \Cref{thm:product_intro} their role is even more pronounced, as the result essentially follows directly from the formulas and the existing result \cite[Corollary 3.4]{van2020limits} giving the probability in \Cref{thm:product_intro} in terms of skew Hall-Littlewood polynomials. 

\subsection{Outline.}
In \Cref{sec:prelim} we set up notation concerning Hall-Littlewood polynomials, and in \Cref{sec:skew_formulas} we prove formulas for principally specialized skew Hall-Littlewood polynomials. These form the main tool for the classification of $\partial \G$ in \Cref{sec:branching_graphs}. In \Cref{sec:infinite_matrices} we explain the setup of $p$-adic random matrix theory, prove an augmented boundary result (\Cref{thm:double_boundary}) tailored to this situation, and use it to prove \Cref{thm:recover_BQ} and \Cref{thm:recover_assiotis}. Finally, in \Cref{sec:matrix_products} we prove \Cref{thm:product_intro}. Finally, in \Cref{sec:appendix_markov} we prove a result about Markov dynamics on $\partial \G$ which is motivated by a parallel work \cite{van2021q}.

\addtocontents{toc}{\protect\setcounter{tocdepth}{1}}
\subsection*{Acknowledgements}  I am grateful to Alexei Borodin for many helpful conversations throughout the project and detailed feedback on several drafts, Theo Assiotis and Alexander Bufetov for comments and suggestions, Grigori Olshanski for discussions on branching graphs, Vadim Gorin for feedback and the suggestion to recover results of \cite{assiotis2020infinite}, Nathan Kaplan, Hoi Nguyen, and Melanie Matchett Wood for discussions on cokernels of matrix products, and the anonymous referee for helpful comments. This material is based on work partially supported by an NSF Graduate Research Fellowship under grant \#$1745302$, and by the NSF FRG grant DMS-1664619.

\section{Hall-Littlewood polynomials} \label{sec:prelim}

In this section we give basic definitions of symmetric functions and Hall-Littlewood polynomials. For a more detailed introduction to symmetric functions see \cite{mac}, and for Macdonald processes see \cite{borodin2014macdonald}.

\subsection{Partitions, symmetric functions, and Hall-Littlewood polynomials.}

We denote by $\cP$ the set of all integer partitions $(\lambda_1,\lambda_2,\ldots)$, i.e. sequences of nonnegative integers $\lambda_1 \geq \lambda_2 \geq \cdots$ which are eventually $0$. We call the integers $\lambda_i$ the \emph{parts} of $\lambda$, set $\lambda_i' = \#\{j: \lambda_j \geq i\}$, and write $m_i(\lambda) = \#\{j: \lambda_j = i\} = \lambda_i'-\lambda_{i+1}'$. We write $\len(\lambda)$ for the number of nonzero parts, and denote the set of partitions of length $\leq n$ by $\cP_n$. We write $\mu \prec \lambda$ or $\lambda \succ \mu$ if $\lambda_1 \geq \mu_1 \geq \lambda_2 \geq \mu_2 \geq \cdots$, and refer to this condition as \emph{interlacing}. 

We denote by $\Lambda_n$ the ring $\C[x_1,\ldots,x_n]^{S_n}$ of symmetric polynomials in $n$ variables $x_1,\ldots,x_n$. It is a very classical fact that the power sum symmetric polynomials $p_k(x_1,\ldots,x_n) = \sum_{i=1}^n x_i^k, k =1,\ldots,n$, are algebraically independent and algebraically generate $\Lambda_n$. For a symmetric polynomial $f$, we will often write $f(\bx)$ for $f(x_1,\ldots,x_n)$ when the number of variables is clear from context. We will also use the shorthand $\bx^\lambda := x_1^{\lambda_1} x_2^{\lambda_2} \cdots x_n^{\lambda_n}$ for $\lambda \in \cP_n$. 

One has a chain of maps
\[
\cdots \to \Lambda_{n+1} \to \Lambda_n \to \Lambda_{n-1} \to \cdots \to 0
\]
where the map $\Lambda_{n+1} \to \Lambda_n$ is given by setting $x_{n+1}$ to $0$. In fact, writing $\Lambda_n^{(d)}$ for symmetric polynomials in $n$ variables of total degree $d$, one has 
\[
\cdots \to \Lambda_{n+1}^{(d)} \to \Lambda_n^{(d)} \to \Lambda_{n-1}^{(d)} \to \cdots \to 0
\]
with the same maps. The inverse limit $\L^{(d)}$ of these systems may be viewed as symmetric polynomials of degree $d$ in infinitely many variables. From the ring structure on each $\Lambda_n$ one gets a natural ring structure on $\Lambda := \bigoplus_{d \geq 0} \L^{(d)}$, and we call this the \emph{ring of symmetric functions}. 
An equivalent definition is $\Lambda := \C[p_1,p_2,\ldots]$ where $p_i$ are indeterminates; under the natural map $\Lambda \to \Lambda_n$ one has $p_i \mapsto p_i(x_1,\ldots,x_n)$. 

Each ring $\Lambda_n$ has a natural basis $\{p_\lambda: \lambda_1 \leq n\}$ where
\begin{equation*}
    p_\lambda := \prod_{i \geq 1} p_{\lambda_i}.
\end{equation*}
Another natural basis, with the same index set, is given by the \emph{Hall-Littlewood polynomials}. Recall the $q$-Pochhammer symbol $(a;q)_n := \prod_{i=0}^{n-1} (1-aq^i)$, and define
\begin{equation*}
    v_\lambda(t) = \prod_{i \in \Z} \frac{(t;t)_{m_i(\lambda)}}{(1-t)^{m_i(\lambda)}}.
\end{equation*}

\begin{defi}\label{def:HL}
The Hall-Littlewood polynomial indexed by $\lambda \in \cP_n$ is
\begin{equation}\label{eq:hlP_formula}
    P_\lambda(\bx;t) = \frac{1}{v_\lambda(t)} \sum_{\sigma \in S_n} \sigma\left(\bx^\lambda \prod_{1 \leq i < j \leq n} \frac{x_i-tx_j}{x_i-x_j}\right)
\end{equation}
where $\sigma$ acts by permuting the variables. We often drop the `$;t$' when clear from context.
\end{defi}

It follows from the definition that 
\begin{equation}\label{eq:P_stable}
    P_\lambda(x_1,\ldots,x_n,0) = P_\lambda(x_1,\ldots,x_n),
\end{equation}
hence for each $\lambda \in \cP$ there is a \emph{Hall-Littlewood symmetric function} $P_\lambda \in \Lambda$.

In another direction, it is desirable to extend these definitions from symmetric polynomials indexed by partitions to symmetric Laurent polynomials indexed by \emph{integer signatures} with possibly negative parts. 
The set of \emph{integer signatures of length $n$} is denoted 
\[
\Sig_n := \{(\lambda_1,\ldots,\lambda_n) \in \Z^n: \lambda_1 \geq \ldots \geq \lambda_n\}.
\]
The integers $\lambda_n$ are called \emph{parts}, as in the case of partitions. We often identify $\lambda \in \cP_n$ with its image in $\Sig_n$ by simply taking the first $n$ parts and forgetting the zeroes which come after. $\Sig_n^{>0} \subset \Sig_n$ is the set of signatures with all parts positive, similarly for $\Sig_n^{\geq 0}$. We set $|\lambda| := \sum_{i=1}^n \lambda_i$ and $m_k(\lambda) = |\{i: \lambda_i = k\}| = \lambda_k' - \lambda_{k+1}'$ as with partitions. For $\lambda \in \Sig_n$ and $\mu \in \Sig_{n-1}$, write $\mu \prec_P \lambda$ if $\lambda_i \geq \mu_i$ and $\mu_i \geq \lambda_{i+1}$ for $1 \leq i \leq n-1$. For $\nu \in \Sig_n$ write $\nu \prec_Q \lambda$ if $\lambda_i \geq \nu_i$ for $1 \leq i \leq n$ and $\nu_i \geq \lambda_{i+1}$ for $1 \leq i \leq n-1$. We write $c[k]$ for the signature $(c,\ldots,c)$ of length $k$, and $()$ for the unique signature of length $0$. We often abuse notation by writing $(\lambda,\mu)$ to refer to the tuple $(\lambda_1,\ldots,\lambda_n,\mu_1,\ldots,\mu_m)$ when $\lambda \in \Sig_n, \mu \in \Sig_m$.

\Cref{def:HL} extends from $\cP_n$ to $\Sig_n$ with no other changes, and we will use the same notation $P_\lambda$ regardless of whether $\lambda$ is a signature or a partition. It is also clear that 
\[
P_{(\lambda_1+1,\ldots,\lambda_n+1)}(\bx;t) = x_1 \cdots x_n P_\lambda(\bx;t).
\]

\begin{defi}
For $\lambda \in \Sig_n$, we define the dual Hall-Littlewood polynomial by
\[
Q_\lambda(\bx;t) = \prod_{i \in \Z} (t;t)_{m_i(\lambda)} P_\lambda(\bx;t).
\]
\end{defi}

We note that in the case where $\lambda \in \cP_n$ has some parts equal to $0$, this normalization is \emph{not} the same as the standard one in e.g. \cite{mac}, though the two agree when $\lambda$ has all parts positive. We use this nonstandard definition because parts equal to $0$ play no special role with integer signatures, though they do in the usual setup with partitions. We will see shortly that the classical results such as branching rules and the Cauchy identity may be stated naturally in this setting.

Because the $P_\lambda$ form a basis for the vector space of symmetric Laurent polynomials in $n$ variables, there exist symmetric Laurent polynomials $P_{\lambda/\mu}(x_1,\ldots,x_{n-k};t) \in \Lambda_{n-k}[(x_1 \cdots x_{n-k})^{-1}]$ indexed by $\lambda \in \Sig_n, \mu \in \Sig_k$ which are defined by
\begin{equation}\label{eq:def_skewP}
    P_\lambda(x_1,\ldots,x_n;t) = \sum_{\mu \in \Sig_k} P_{\lambda/\mu}(x_{k+1},\ldots,x_n;t) P_\mu(x_1,\ldots,x_k;t).
\end{equation}
We define the skew $Q$ functions in a slightly nonstandard way where the lengths of both signatures are the same, in contrast to the skew $P$ functions; this is inspired by the higher spin Hall-Littlewood polynomials introduced in \cite{borodin2017family}. 

\begin{defi}
For $\lambda,\nu \in \Sig_n^{>0}$ and $k \geq 1$ arbitrary, define $Q_{\nu/\lambda}(x_1,\ldots,x_k;t) \in \lambda_k$ by
\begin{equation}\label{eq:def_skewQ}
    Q_{(\nu,0[k])}(x_1,\ldots,x_{n+k};t) = \sum_{\lambda \in \Sig_n^{> 0}} Q_{\nu/\lambda}(x_{n+1},\ldots,x_{n+k};t) Q_\lambda(x_1,\ldots,x_k;t).
\end{equation}
\end{defi}

In particular, $Q_{\lambda/(0[n])}(x_1,\ldots,x_{n+k};t)$ agrees with $Q_{(\lambda,0[k])}(x_1,\ldots,x_{n+k};t)$ as defined earlier, and we will use both interchangeably. Recall the two interlacing relations on signatures $\succ_P, \succ_Q$ defined above.

\begin{defi}\label{def:psi_varphi_coefs}
For $\mu \in \Sig_{n+1}, \lambda,\nu \in \Sig_n$ with $\mu \succ_P \lambda, \nu \succ_Q \lambda$, let
\begin{equation*}\label{eq:pbranch}
    \psi_{\mu/\lambda} :=  \prod_{\substack{i \in \Z \\ m_i(\lambda) = m_i(\mu)+1}} (1-t^{m_i(\lambda)}) 
\end{equation*}
and
\begin{equation*}\label{eq:qbranch}
    \varphi_{\nu/\lambda} :=  \prod_{\substack{i \in \Z \\  m_i(\nu) = m_i(\lambda)+1}} (1-t^{m_i(\nu)}) 
\end{equation*}
\end{defi}

The following branching rule is standard, but in this specific formulation with signatures follows from \cite[Lemma 2.1 and Proposition 2.8]{van2020limits}.

\begin{lem} \label{thm:branching_formulas}
For $\lambda,\nu \in \Sig_{n}^{>0}, \mu \in \Sig_{n-k}^{>0}$, we have
\begin{equation}\label{eq:skewP_branch_formula}
    P_{\lambda/\mu}(x_1,\ldots,x_k) = \sum_{\mu = \lambda^{(1)} \prec_P \lambda^{(2)} \prec_P \cdots \prec_P \lambda^{(k)}= \lambda} \prod_{i=1}^{k-1} x_i^{|\lambda^{(i+1)}|-|\lambda^{(i)}|}\psi_{\lambda^{(i+1)}/\lambda^{(i)}}
\end{equation}
and
\begin{equation}\label{eq:skewQ_branch_formula}
    Q_{\lambda/\nu}(x_1,\ldots,x_k) = \sum_{\nu = \lambda^{(1)} \prec_Q \lambda^{(2)} \prec_Q \cdots \prec_Q \lambda^{(k)}=\lambda} \prod_{i=1}^{k-1} x_i^{|\lambda^{(i+1)}|-|\lambda^{(i)}|}\varphi_{\lambda^{(i+1)}/\lambda^{(i)}}.
\end{equation}
\end{lem}

The formulas from \Cref{thm:branching_formulas} may be used to define skew functions for general signatures. 

\begin{defi}\label{def:skew_fns}
For $\lambda,\nu \in \Sig_{n}, \mu \in \Sig_{n-k}$, define $P_{\lambda/\mu}(x_1,\ldots,x_k)$ and $Q_{\lambda/\nu}(x_1,\ldots,x_k)$ by the formulas \eqref{eq:skewP_branch_formula} and \eqref{eq:skewQ_branch_formula} respectively.
\end{defi}

It follows from \eqref{eq:P_stable} and \eqref{eq:def_skewP} that for $\lambda \in \Sig_n^{\geq 0},\mu \in \Sig_k^{\geq 0}$,
\begin{equation}\label{eq:skewP_add_zeros}
    P_{\lambda/\mu}(x_1,\ldots,x_{n-k}) = P_{(\lambda,0)/(\mu,0)}(x_1,\ldots,x_{n-k}) = P_{(\lambda,0)/\mu}(x_1,\ldots,x_{n-k},0).
\end{equation}
Therefore there exists a symmetric function $P_{(\lambda,0,\ldots)/(\mu,0,\ldots)} \in \Lambda$ associated to the pair of partitions $(\lambda,0,\ldots),(\mu,0,\ldots) \in \Y$, which maps to 
\[
 P_{\lambda/\mu}(x_1,\ldots,x_{n-k}) = P_{(\lambda,0)/(\mu,0)}(x_1,\ldots,x_{n-k})
\]
under the map $\Lambda \to \Lambda_{n-k}$. With $Q$ the situation is slightly more subtle: given either $\lambda,\nu \in \Sig_n$ or $\lambda,\nu \in \Y$ there exists an element $Q_{\nu/\lambda} \in \Lambda$. If additionally $\lambda,\nu \in \Sig_n^{\geq 0}$ then in fact $Q_{\nu/\lambda} = Q_{(\nu,0)/(\lambda,0)}$. These properties can all be checked from the above. 

\begin{rmk}\label{rmk:P_trans_inv}
It follows from \Cref{def:HL} and \Cref{def:skew_fns} that for any $D \in \Z$,
\begin{align*}
    P_{\lambda+D[n]}(x_1,\ldots,x_n) &= (x_1 \cdots x_n)^D P_{\lambda}(x_1, \ldots, x_n) \\
    P_{(\mu+D[m])/(\lambda+D[n])}(x_1,\ldots,x_{m-n}) &= (x_1 \cdots x_{m-n})^DP_{\mu/\lambda}(x_1,\ldots,x_{m-n}) \\
\end{align*}
\end{rmk}

We note that $P_{\lambda/\mu}(x_1,\ldots,x_k)$ is in general a Laurent polynomial, while $Q_{\lambda/\nu}(x_1,\ldots,x_k)$ is always a polynomial.

Hall-Littlewood polynomials satisfy the \emph{skew Cauchy identity}, upon which most probabilistic constructions rely. A formulation in terms of signatures is given for the more general Macdonald polynomials in \cite[Lemma 2.3]{van2020limits}, and the below statement follows immediately by specializing that one.

\begin{prop}\label{thm:finite_cauchy}
Let $\nu \in \Sig_k, \mu \in \Sig_{n+k}$. Then
\begin{multline}\label{eq:finite_cauchy}
    \sum_{\kappa \in \Sig_{n+k}} P_{\kappa/\nu}(x_1,\ldots,x_n;t)Q_{\kappa/\mu}(y_1,\ldots,y_m;t) \\
    = \prod_{\substack{1 \leq i \leq n \\ 1 \leq j \leq m}} \frac{1-tx_iy_j}{1-x_iy_j} \sum_{\lambda \in \Sig_k} Q_{\nu/\lambda}(y_1,\ldots,y_m;t) P_{\mu/\lambda}(x_1,\ldots,x_n;t).
\end{multline}
\end{prop}

For later convenience we set
\begin{equation}\label{eq:def_cauchy_kernel}
     \Pi(\bx;\by) := \prod_{\substack{1 \leq i \leq n \\ 1 \leq j \leq m}} \frac{1-tx_iy_j}{1-x_iy_j} = \exp\left(\sum_{\ell = 1}^\infty \frac{1-t^\ell}{\ell}p_\ell(\bx)p_\ell(\by)\right)
\end{equation}
(The second equality in \eqref{eq:def_cauchy_kernel} is not immediate but is shown in \cite{mac}).

There is also a more standard form of the Cauchy identity with integer partitions rather than signatures, see \cite{mac}: For $\mu,\nu \in \Y$, 
\begin{multline}\label{eq:infinite_cauchy}
    \sum_{\kappa \in \Y} P_{\kappa/\nu}(x_1,\ldots,x_n;t)Q_{\kappa/\mu}(y_1,\ldots,y_m;t) \\
    = \prod_{\substack{1 \leq i \leq n \\ 1 \leq j \leq m}} \frac{1-tx_iy_j}{1-x_iy_j} \sum_{\lambda \in \Y} Q_{\nu/\lambda}(y_1,\ldots,y_m;t) P_{\mu/\lambda}(x_1,\ldots,x_n;t).
\end{multline}

Hall-Littlewood polynomials/functions may be used to define Markovian dynamics on $\Sig_n$. Given finite or infinite sequences $\ba, \mathbf{b}$ of nonnegative real $a_i,b_j$ with finite sums and $a_ib_j < 1$, the expressions
\begin{equation}\label{eq:general_cauchy_dynamics}
    \Pr(\lambda \to \nu) := Q_{\nu/\lambda}(\ba)\frac{P_\nu(\mathbf{b})}{P_\lambda(\mathbf{b}) \Pi(\mathbf{a}; \mathbf{b})}
\end{equation}
define transition probabilities by \Cref{thm:finite_cauchy}. The joint distribution of such dynamics, run for $k$ steps with possibly distinct specializations $\ba^{(1)},\ldots,\ba^{(k)}$ and started at $0[n]$, is a so-called \emph{ascending Hall-Littlewood process}. It is a measure on $\Sig_n^k$ given by 
\begin{equation}\label{eq:general_hl_proc}
    \Pr(\lambda^{(1)},\ldots,\lambda^{(k)}) = \frac{P_{\lambda^{(k)}}(\mathbf{b}) \prod_{i=1}^k Q_{\lambda^{(i)}/\lambda^{(i-1)}}(\ba^{(i)}) }{\Pi(\mathbf{b}; \ba^{(1)},\ldots,\ba^{(k)})}
\end{equation}
where we take $\lambda^{(0)} = (0[n])$.

Finally, there are simple explicit formulas for the Hall-Littlewood polynomials when a geometric progression $u,ut,\ldots,ut^{n-1}$ is substituted for $x_1,\ldots,x_n$--this is often referred as a \emph{principal specialization}. For $\lambda \in \Sig_n$ let
\begin{equation}\label{eq:n(lambda)}
    n(\lambda) := \sum_{i=1}^n (i-1)\lambda_i,
\end{equation}
and note that if additionally $\lambda \in \Sig_n^{\geq 0}$ then 
\[
n(\lambda) = \sum_{x \geq 1} \binom{\lambda_x'}{2}.
\]
The following formula is standard. It may be easily derived from \eqref{eq:hlP_formula} by noting that the summand is zero unless the permutation is the identity, and evaluating this summand.

\begin{prop}[Principal specialization formula]\label{thm:hl_principal_formulas}
For $\lambda \in \Sig_n$,
\begin{align*}
    P_\lambda(u,ut,\ldots,ut^{n-1};t) &= u^{|\lambda|} t^{n(\lambda)} \frac{(t;t)_n}{\prod_{i \in \Z} (t;t)_{m_i(\lambda)}}.
\end{align*}
\end{prop}

The proof sketched above offers no clear extension to the case of skew polynomials, but in \Cref{sec:skew_formulas} we will derive such formulas using recent results of \cite{borodin2018higher_original_paper}.

\section{Principally specialized skew Hall-Littlewood polynomials} \label{sec:skew_formulas}

In this section we prove \eqref{eq:inf_skew_princ_intro} from the Introduction and its analogue for $Q_{\nu/\lambda}$ in \Cref{thm:skew_formula}, as well as extensions when the geometric progression is finite and the formulas are less simple in \Cref{thm:finite_skew_fn_computation}. Let us introduce a bare minimum of background on higher spin Hall-Littlewood polynomials $F_{\mu/\lambda}, G_{\nu/\lambda}$, which generalize the usual Hall-Littlewood polynomials $P,Q$ by the addition of an extra parameter $s$. We omit their definition, which may be found in \cite{borodin2017family,borodin2017vertex_lecture_notes}, as we will only care about the case $s=0$ when they reduce to slightly renormalized Hall-Littlewood polynomials. When $s=0$, for $\lambda,\nu \in \Sig_{n}^{\geq 0}, \mu \in \Sig_{n+k}^{\geq 0}$ one has 
\begin{equation}\label{eq:FtoP}
    F_{\mu/\lambda}(x_1,\ldots,x_k)\Big\vert_{s=0} = \prod_{i \geq 0} \frac{(t;t)_{m_i(\mu)}}{(t;t)_{m_i(\lambda)}} P_{\mu/\lambda}(x_1,\ldots,x_k)
\end{equation}
and 
\begin{equation}
    G_{\nu/\lambda}(x_1,\ldots,x_k)\Big\vert_{s=0} = Q_{\nu/\lambda}(x_1,\ldots,x_k)
\end{equation}
by \cite[\S 8.1]{borodin2017family}. Formulas for principally specialized skew $F$ and $G$ functions were shown in \cite{borodin2017family}, though we will state the version given later in \cite{borodin2017vertex_lecture_notes}. We apologize to the reader for giving a formula for an object which we have not actually defined, but will immediately specialize to the Hall-Littlewood case, so we hope no confusion arises. We need the following notation.

\begin{defi}\label{def:qhyp}
The normalized terminating $q$-hypergeometric function is
\begin{equation}\label{eq:qhyp_def}
    _{r+1}\bar{\phi}_r \left(\begin{matrix} 
t^{-n};a_1,\ldots,a_r \\ b_1,\ldots,b_r \end{matrix} 
; t,z \right) := \sum_{k=0}^n z^k \frac{(t^{-n};t)_k}{(t;t)_k} \prod_{i=1}^r(a_i;t)_k (b_it^k;t)_{n-k}
\end{equation}
for $n \in \Z_{\geq 0}$ and $|z|,|t|<1$. 
\end{defi}

\begin{prop}[{\cite[Proposition 5.5.1]{borodin2017vertex_lecture_notes}}]\label{thm:formula_from_bf}
Let $J \in \Z_{\geq 1},\lambda \in \Sig_n^{\geq 0}, \mu \in \Sig_{n+J}^{\geq 0}$. Then
\begin{equation}\label{eq:xproduct_general_s}
    F_{\mu/\lambda}(u,tu,\ldots,t^{J-1}u) = \prod_{x \in \Z_{\geq 0}} w_u^{(J)}(i_1(x),j_1(x);i_2(x),j_2(x)),
\end{equation}
where the product is over the unique collection of $n+J$ up-right paths on the semi-infinite horizontal strip of height $1$ with paths entering from the bottom at positions $\lambda_i, 1 \leq i \leq n$, $J$ paths entering from the left, and paths exiting from the top at positions $\mu_i, 1 \leq i \leq n+J$, see \Cref{fig:path_example}. Here $i_1(x),j_1(x),i_2(x),j_2(x)$ are the number of paths on the south, west, north and east edge of the vertex at position $x$ as in \Cref{fig:ij_notation}, and the weights in the product are given by
\begin{multline}\label{eq:weights_with_s}
     w_u^{(J)}(i_1,j_1;i_2,j_2) := \delta_{i_1+j_1,i_2+j_2}\frac{(-1)^{i_1+j_2}t^{\frac{1}{2}i_1(i_1+2j_1-1)}s^{j_2-i_1}u^{i_1}(t;t)_{j_1}(us^{-1};t)_{j_1-i_2}}{(t;t)_{i_1}(t;t)_{j_2}(us;t)_{i_1+j_1}} \\
     \times \qhyp
     \left(\begin{matrix} 
t^{-i_1};t^{-i_2},t^Jsu,tsu^{-1} \\ 
s^2,t^{1+j_1-i_2},t^{1+J-i_1-j_1} \end{matrix} 
; t,t \right).
\end{multline}
Similarly, for $\lambda,\nu \in \Sig_n^{\geq 0}$, $G_{\nu/\lambda}(u,tu,\ldots,t^{J-1}u)$ is given by the product of the same weights over the unique collection of $n$ up-right paths on the same strip entering from the bottom at positions $\lambda_i, 1 \leq i \leq n$ and exiting from the top at positions $\nu_i, 1 \leq i \leq n$. 
\end{prop}

\bigskip

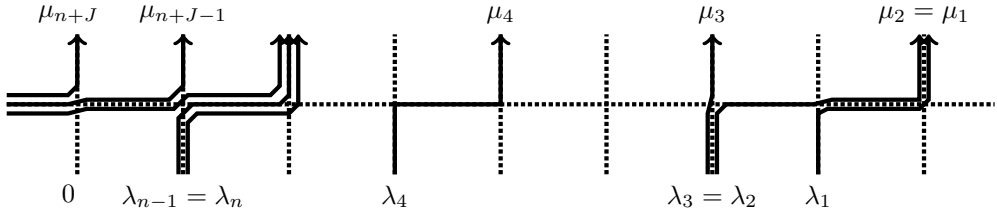
\begin{figure}[htbp]
\centering
	\scalebox{1}{\begin{tikzpicture}
	[scale=1.16,ultra thick]
	\def\d{.105}
	\def\h{1.2} 
	\draw[densely dotted] (-.8,0)--++(8.8+2*\h,0);
	\foreach \ii in {0,1,2,3,4,5,6,7,8}
	{
		\draw[densely dotted] (\h*\ii,-.8)--++(0,1.6);
	}
	\node[below] at (-.1,-.8) {0};
	\node[below] at (\h,-.8) {$\lambda_{n-1}=\lambda_n$};
	\node[below] at (6*\h,-.8) {$\lambda_{3}=\lambda_{2}$};
	\node[below] at (7*\h,-.8) {$\lambda_{1}$};
	\node[below] at (3*\h,-.8) {$\lambda_{4}$};
	\node[above] at (8*\h,.8) {$\mu_{2}=\mu_1$};
	\node[above] at (6*\h,.8) {$\mu_3$};
	\node[above] at (4*\h,.8) {$\mu_4$};
	\node[above] at (0*\h-.1,.8) {$\mu_{n+J}$};
	\node[above] at (1*\h,.8) {$\mu_{n+J-1}$};
	\draw[->] (7*\h,-.8)--++(0,.8-\d)--++(\d,\d/2)--++(\h-\d*3/2,0)--++(\d,\d)--++(0,.8-\d/2);
	\draw[->] (6*\h+\d/2,-.8)--++(0,.8-\d)--++(\d,\d)--++(\h-2*\d,0)--++(2*\d,\d/2)--++(\h-2*\d,0)--++(0,.8-\d/2);
	\draw[->] (6*\h-\d/2,-.8)--++(0,.8-\d)--++(\d/2,2*\d)--++(0,.8-\d);
	\draw[->] (3*\h,-.8)--++(0,.8)--++(\h,0)--++(0,.8);
	\draw[->] (-.8,\d)--++(.8-\d,0)--++(\d,\d)--++(0,.8-2*\d);
	\draw[->] (-.8,0)--++(.8-\d,0)--++(2*\d,\d/2)--++(\h-\d*5/2,0)--++(3/2*\d,3/2*\d)--++(0,.8-2*\d);
	\draw[->] (-.8,-\d)--++(.8-\d,0)--++(2*\d,\d/2)--++(\h-2*\d,0)--++(3/2*\d,3/2*\d)--++(\h-\d/2-\d-\d,0)--++(\d,\d)
	--++(0,.8-2*\d);
	\draw[->] (\h-\d/2,-.8)--++(0,.8-3/2*\d)--++(3/2*\d,3/2*\d)--++(\h-2*\d,0)--++(\d,\d)--++(0,.8-\d);
	\draw[->] (\h+\d/2,-.8)--++(0,.8-\d*2)--++(\d,\d)--++(\h-3/2*\d,0)--++(\d,\d)--++(0,.8);
\end{tikzpicture}}
\bigskip
\caption{
	The unique path collection corresponding to the function 
	$\F_{\mu/\lambda}(u,qu,\ldots,q^{J-1}u)$ with $J=3,n=6,\lambda=(7,6,6,4,1,1),\mu = (8,8,6,4,2,2,2,1,0)$.}
\label{fig:path_example}
\end{figure}

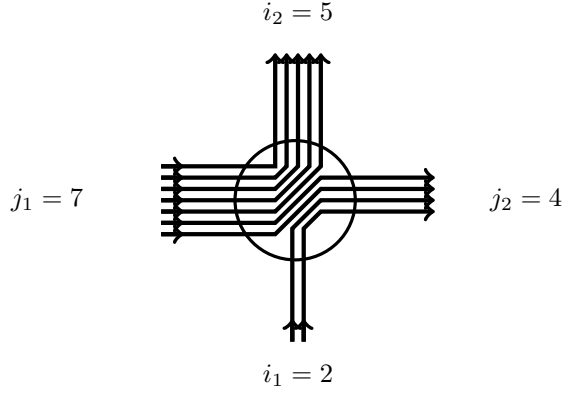
\begin{figure}[htbp]
\bigskip
\begin{center}
\begin{tikzpicture}
		[scale=1.5, ultra thick]
		\def\d{.1}
		\draw[->] (-1,3*\d)--++(1,0)--++(0,1);
		\draw[->] (-1,2*\d)--++(1,0)--++(\d,\d)--++(0,1);
		\draw[->] (-1,\d)--++(1,0)--++(2*\d,2*\d)--++(0,1);
		\draw[->] (-1,0)--++(1,0)--++(3*\d,3*\d)--++(0,1);
		\draw[->] (-1,-\d)--++(1,0)--++(4*\d,4*\d)--++(0,1);
		\draw[->] (-1,-2*\d)--++(1,0)--++(4*\d,4*\d)--++(1,0);
		\draw[->] (-1,-3*\d)--++(1,0)--++(4*\d,4*\d)--++(1,0);
		\draw[->] (1.5*\d,-1-2.5*\d)--++(0,1)--++(2.5*\d,2.5*\d)--++(1,0);
		\draw[->] (2.5*\d,-1-2.5*\d)--++(0,1)--++(1.5*\d,1.5*\d)--++(1,0);
		\node at (.2,-1.55) {$i_1=2$};
		\node at (-2,0) {$j_1=7$};
		\node at (.2,1.65) {$i_2=5$};
		\node at (2.2,0) {$j_2=4$};
		\draw[->] (-1,3*\d)--++(.2,0);1
		\draw[->] (-1,2*\d)--++(.2,0);
		\draw[->] (-1,1*\d)--++(.2,0);
		\draw[->] (-1,0*\d)--++(.2,0);
		\draw[->] (-1,-1*\d)--++(.2,0);
		\draw[->] (-1,-2*\d)--++(.2,0);
		\draw[->] (-1,-3*\d)--++(.2,0);
		\draw[->] (1.5*\d,-1-2.5*\d)--++(0,.2);
		\draw[->] (2.5*\d,-1-2.5*\d)--++(0,.2);
		\draw[very thick] (1.75*\d,0) circle (15pt);
\end{tikzpicture}
\end{center}
\caption{Illustration of the notation for edges, in the example $(i_1,j_1;i_2,j_2)=(2,7;5,4)$.} 
	\label{fig:ij_notation}
\end{figure}

\begin{rmk}
To avoid confusion with \cite{borodin2017family,borodin2017vertex_lecture_notes}, we note that the parameter which we call $t$ for consistency with Hall-Littlewood notation is denoted by $q$ in these references.
\end{rmk}
We now introduce some notation and specialize \Cref{thm:formula_from_bf} to the Hall-Littlewood case $s=0$. 

\begin{defi}\label{def:gen:n(lambda)}
For $\lambda,\nu \in \Sig_n$, we define 
\[
n(\nu/\lambda) := \sum_{\substack{1 \leq i  < j \leq n}} \max(\nu_j - \lambda_i, 0) = \sum_{x \geq \lambda_n} \binom{\nu_{x+1}' - \lambda_{x+1}'}{2}.
\]
We additionally allow the case when $\lambda,\nu \in \Y$; the first formula makes sense with the $\leq n$ removed, while for the second we simply replace the sum over $x \geq \lambda_n$ by $x \geq 0$. 
\end{defi}

Note that $n(\nu/\lambda)$
\begin{enumerate}
    \item is translation-invariant, $n((\nu+D[n])/(\lambda+D[n])) = n(\nu/\lambda)$, and
    \item generalizes the standard definition of $n(\nu)$ in \eqref{eq:n(lambda)}, namely when $\nu \in \Sig_n^{\geq 0}$ then $n(\nu) = n(\nu/(0[n]))$.
\end{enumerate}
One may also view $n(\nu/\lambda)$ as quantifying the failure of $\nu$ and $\lambda$ to interlace; it is $0$ when $\nu,\lambda$ interlace, and increases by $1$ when a part of $\lambda$ is moved past a part of $\nu$.

\begin{prop}\label{thm:finite_skew_fn_computation}
For $J \in \Z_{\geq 1},\lambda \in \Sig_n^{\geq 0}, \mu \in \Sig_{n+J}^{\geq 0}$,
\begin{align}\label{eq:finite_skew_formula}
P_{\mu/\lambda}(u,\ldots,ut^{J-1}) =
     (t;t)_J u^{|\mu|-|\lambda|} \prod_{x \geq 0} \frac{t^{m_x(\lambda)m_x(\mu)+\binom{\mu_{x+1}'-\lambda_{x+1}'}{2}}}{(t;t)_{m_x(\mu)}}\qhypm
     \left(\begin{matrix} 
t^{-m_x(\lambda)};t^{-m_x(\mu)},0 \\ 
t^{1+\mu_{x+1}'-\lambda_x'},t^{1+J-\mu_x'+\lambda_{x+1}'}\end{matrix} 
; t,t \right).
\end{align}
For $\lambda,\nu \in \Sig_n$,
\begin{align}\label{eq:finite_skew_Q}
    Q_{\nu/\lambda}(u,\ldots,ut^{J-1}) = u^{|\nu|-|\lambda|} t^{n(\nu/\lambda)} \prod_{x \in \Z} \frac{t^{m_x(\lambda)m_x(\nu)}}{(t;t)_{m_x(\lambda)}} \qhypm
     \left(\begin{matrix} 
t^{-m_x(\lambda)};t^{-m_x(\nu)},0 \\ 
t^{1+\nu_{x+1}'-\lambda_x'},t^{1+J-\nu_x'+\lambda_{x+1}'}\end{matrix} 
; t,t \right).
\end{align}
\end{prop}
\begin{proof}
We begin with \eqref{eq:finite_skew_formula}. 
In this case we may apply \Cref{thm:formula_from_bf} to compute 
\begin{equation}\label{eq:switch_to_F_finite}
    \text{LHS\eqref{eq:finite_skew_formula}} = F_{\mu/\lambda}(u,\ldots,ut^{J-1})\Big \vert_{s=0} \prod_{i \geq 0} \frac{(t;t)_{m_i(\lambda)}}{(t;t)_{m_i(\mu)}}.
\end{equation}
When $s \to 0$, the factor $s^{j_2-i_1}(us^{-1};t)_{j_1-i_2}$ in \eqref{eq:weights_with_s} converges to $(-u)^{j_1-i_2}t^{\binom{j_1-i_2}{2}}$ (using that $j_2-i_1 = j_1-i_2$). The sign cancels with the sign in \eqref{eq:weights_with_s}, and the power of $u$ combines with the $u^{i_1}$ in \eqref{eq:weights_with_s} to give $u^{j_2}$, so \eqref{eq:weights_with_s} becomes
\begin{equation}\label{eq:weights_with_s=0}
     w_{t^n}^{(J)}(i_1,j_1;i_2,j_2) = \delta_{i_1+j_1,i_2+j_2}u^{j_2}\frac{t^{\frac{1}{2}i_1(i_1+2j_1-1)+\binom{j_1-i_2}{2}}(t;t)_{j_1}}{(t;t)_{i_1}(t;t)_{j_2}} \qhyp
     \left(\begin{matrix} 
t^{-i_1};t^{-i_2},0,0 \\ 
0,t^{1+j_1-i_2},t^{1+J-i_1-j_1} \end{matrix} 
; t,t \right).
\end{equation}
In the product \eqref{eq:xproduct_general_s} when the weights are specialized to \eqref{eq:weights_with_s=0}, some of the factors simplify, as
\begin{equation}
    \prod_{x \geq 0} \frac{(t;t)_{j_1}}{(t;t)_{i_1}(t;t)_{j_2}} = \frac{(t;t)_J}{\prod_{x \in \Z} (t;t)_{i_1(x)}}
\end{equation}
because the $\frac{(t;t)_{j_1(x)}}{(t;t)_{j_2(x)}}$ factor cancels except for a $(t;t)_J$ from the paths incoming from the left. Hence
\begin{equation}\label{eq:final_finite_weight_product}
    \prod_{x \geq 0} w_{u}^{(J)}(i_1,j_1;i_2,j_2) = (t;t)_J \prod_{x \geq 0} u^{j_2} \frac{t^{\frac{1}{2}i_1(i_1+2j_1-1)+\binom{j_1-i_2}{2}}}{(t;t)_{m_x(\lambda)}} \qhypm
     \left(\begin{matrix} 
t^{-i_1};t^{-i_2},0 \\ 
t^{1+j_1-i_2},t^{1+J-i_1-j_1} \end{matrix} 
; t,t \right).
\end{equation}
Using that $j_2=i_1+j_1-i_2$ simplifies the exponent of $t$ in \eqref{eq:final_finite_weight_product} to
\[
\frac{1}{2}i_1(i_1+2j_1-1)+\binom{j_1-i_2}{2} = \binom{j_2}{2} + i_1i_2.
\]
To convert to the form in terms of partitions, we record the following translations between the $i$'s and $j$'s and the usual conjugate partition notation: 
\begin{align}\label{eq:ijs_to_partitions}
\begin{split}
    i_1(x) &= \lambda_x'-\lambda_{x+1}' = m_x(\lambda), \\
    j_1(x) &= \mu_x'-\lambda_x', \\
    i_2(x) &= \mu_x'-\mu_{x+1}' = m_x(\mu), \\
    j_2(x) &= \mu_{x+1}'-\lambda_{x+1}'.
\end{split}
\end{align}
Translating \eqref{eq:final_finite_weight_product} into partition notation and multiplying by the $ \prod_{i \geq 0} \frac{(t;t)_{m_i(\lambda)}}{(t;t)_{m_i(\mu)}}$ factor of \eqref{eq:switch_to_F_finite} yields \eqref{eq:finite_skew_formula}.

To prove \eqref{eq:finite_skew_Q} we first note that both sides of \eqref{eq:finite_skew_Q} are translation-invariant, so without loss of generality we may take $\lambda,\nu \in \Sig_n^{> 0}$. We then likewise appeal to \Cref{thm:formula_from_bf} and either make the same argument as above or deduce it from \eqref{eq:finite_skew_formula} by considering $F_{(\nu,0[J])/\lambda}(u,\ldots,ut^{J-1})$ for $\lambda,\nu \in \Sig_n^{> 0}$ and using the relation between $P$ and $Q$ polynomials. Since $\lambda,\nu$ are of the same length we have
\[
\prod_{x \geq 0} t^{\binom{j_2(x)}{2}} = t^{n(\nu/\lambda)}
\]
by \eqref{eq:ijs_to_partitions}. Finally, note that the product can be extended from $x \geq 0$ to $x \in \Z$, which in this translation-invariant setting is more aesthetically appealing.
\end{proof}

\begin{rmk}
While it follows from the branching rule that for nonnegative signatures $\mu,\lambda$ of appropriate lengths,
\[
P_{(\mu,0)/(\lambda,0)}(u,\ldots,ut^{J-1}) = P_{\mu/\lambda}(u,\ldots,ut^{J-1}),
\]
see \eqref{eq:skewP_add_zeros}, this relation is not readily apparent from \eqref{eq:finite_skew_formula}. The only term on the RHS of \eqref{eq:finite_skew_formula} which a priori might differ after padding $\lambda,\mu$ with zeros is the $x=0$ term of the product. It may be checked that this term is in fact unchanged by padding with zeros, but this is not immediately obvious from the formula as written.
\end{rmk}

The next result takes the $J \to \infty$ limit of \Cref{thm:finite_skew_fn_computation}. Recall from \Cref{sec:prelim} that if $\mu,\lambda \in \Y$, $P_{\mu/\lambda} \in \Lambda$ is a polynomial in the power sums $p_1,p_2,\ldots$. Hence given any infinite sequence of complex numbers $a_1,a_2,\ldots$ such that the sums $p_k(a_1,a_2,\ldots)$ converge (it suffices for this to hold for $p_1$), we may define $P_{\mu/\lambda}(a_1,a_2,\ldots) \in \C$. This is how $P_{\mu/\lambda}(u,ut,\ldots)$ is to be interpreted below, and similarly for $Q_{\nu/\lambda}(u,ut,\ldots)$.

\begin{thm}\label{thm:skew_formula}
For $\mu,\lambda \in \Y$, we have
\begin{align}\label{eq:princ_skew_P}
\begin{split}
P_{\mu/\lambda}(u,ut,\ldots) &= u^{|\mu|-|\lambda|} t^{n(\mu/\lambda)}\prod_{x \geq 1} \frac{(t^{1+\mu_x'-\lambda_x'};t)_{m_x(\lambda)}}{(t;t)_{m_x(\mu)}}.
\end{split}
\end{align}
For $\lambda,\nu \in \Sig_n$, 
\begin{align}\label{eq:princ_skew_Q}
    Q_{\nu/\lambda}(u,ut,\ldots) = u^{|\nu|-|\lambda|} t^{n(\nu/\lambda)}\prod_{x \in \Z} \frac{(t^{1+\nu_x'-\lambda_x'};t)_{m_x(\lambda)}}{(t;t)_{m_x(\lambda)}} 
\end{align}
\end{thm}

\begin{proof}
For $n \geq \len(\lambda), n+J \geq \len(\mu)$, we may identify $\mu,\lambda \in \Y$ with nonnegative signatures $\mu(n+J) \in \Sig_{n+J}^{\geq 0}, \lambda(n) \in \Sig_n^{\geq 0}$ given by truncating. Hence to compute
\[
P_{\mu/\lambda}(u,ut,\ldots)
\]
it suffices to take $J \to \infty$ in \eqref{eq:finite_skew_formula}. The polynomial $P_{\mu(n+J)/\lambda(n)}$ is independent of $n$ for all $n$ sufficiently large, see \eqref{eq:skewP_add_zeros}, so we will fix $n$ and will abuse notation below and write $\lambda$ for $\lambda(n)$. We first pull the $1/(t;t)_{m_x(\mu)}$ out of the product, and note that $m_0(\mu(n+J)) \to (t;t)_\infty$ as $J \to \infty$, cancelling the $(t;t)_J$ term of \eqref{eq:finite_skew_formula}. We write the remaining term inside the product in \eqref{eq:finite_skew_formula} as 
\begin{equation}\label{eq:split_product}
    \left(t^{m_x(\lambda)m_x(\mu(n+J))}\qhypm
     \left(\begin{matrix} 
t^{-m_x(\lambda)};t^{-m_x(\mu(n+J))},0 \\ 
t^{1+\mu(n+J)_{x+1}'-\lambda_x'},t^{1+J-\mu(n+J)_x'+\lambda_{x+1}'}\end{matrix} 
; t,t \right)\right).
\end{equation}
To show \eqref{eq:princ_skew_P} it suffices to show that for $x > 0$,
\begin{equation}\label{eq:qhyp_convergence}
    \lim_{J \to \infty} t^{m_x(\lambda)m_x(\mu(n+J))}\qhypm
     \left(\begin{matrix} 
t^{-m_x(\lambda)};t^{-m_x(\mu(n+J))},0 \\ 
t^{1+\mu(n+J)_{x+1}'-\lambda_x'},t^{1+J-\mu(n+J)_x'+\lambda_{x+1}'}\end{matrix} 
; t,t \right) = (t^{1+\mu_x'-\lambda_x'};t)_{m_x(\lambda)}.
\end{equation}
and for $x=0$,
\begin{equation}\label{eq:qhyp_convergence_degenerate}
     \lim_{J \to \infty} t^{m_x(\lambda)m_x(\mu(n+J))}\qhypm
     \left(\begin{matrix} 
t^{-m_x(\lambda)};t^{-m_x(\mu(n+J))},0 \\ 
t^{1+\mu(n+J)_{x+1}'-\lambda_x'},t^{1+J-\mu(n+J)_x'+\lambda_{x+1}'}\end{matrix} 
; t,t \right) = 1.
\end{equation}
We begin with \eqref{eq:princ_skew_P}. Then $1+J-\mu(n+J)_x'+\lambda_{x+1}' \to \infty$ and all other arguments in the $q$-hypergeometric function remain the same, so the LHS of \eqref{eq:qhyp_convergence} is 
\begin{equation}\label{eq:qhyp_J_infty}
    \qhypl
     \left(\begin{matrix} 
t^{-m_x(\lambda)};t^{-m_x(\mu)} \\ 
t^{1+\mu_{x+1}'-\lambda_x'} \end{matrix} 
; t,t \right) = \sum_{\ell=0}^{m_x(\lambda)} t^\ell \frac{(t^{-m_x(\lambda)};t)_\ell}{(t;t)_\ell}(t^{-m_x(\mu)};t)_\ell(t^{1+\mu_{x+1}'-\lambda_x'+\ell};t)_{m_x(\lambda)-\ell}.
\end{equation}
To apply known identities, we reexpress the above in terms of the more standard terminating $q$-hypergeometric series $\;_2 \phi_1$ as
\begin{multline}\label{eq:change_qhyp_normalization}
   (t^{1+\mu_{x+1}'-\lambda_x'};t)_{m_x(\lambda)} \sum_{\ell=0}^{m_x(\lambda)} t^\ell \frac{(t^{-m_x(\lambda)};t)_\ell(t^{-m_x(\mu)};t)_\ell}{(t;t)_\ell (t^{1+\mu_{x+1}'-\lambda_x'};t)_\ell} \\ = (t^{1+\mu_{x+1}'-\lambda_x'};t)_{m_x(\lambda)} \;_2 \phi_1  \left(\begin{matrix} 
t^{-m_x(\lambda)};t^{-m_x(\mu)} \\ 
t^{1+\mu_{x+1}'-\lambda_x'} \end{matrix} 
; t,t \right). 
\end{multline}
By a special case of the $q$-Gauss identity, see e.g. \cite[Exercise 3.17]{koepf1998hypergeometric},
\begin{equation}\label{eq:from_koepf}
    _2\phi_1 \left(\begin{matrix} 
t^{-n}; b \\ 
c \end{matrix} 
; t,t \right) = \frac{(c/b;t)_n}{(c;t)_n}b^n.
\end{equation}
Applying \eqref{eq:from_koepf} with $b=t^{-m_x(\mu)},c=t^{1+\mu_{x+1}'-\lambda_x'}$ to \eqref{eq:change_qhyp_normalization} yields
\begin{equation}\label{eq:apply_koepf}
    \qhypl
     \left(\begin{matrix} 
t^{-m_x(\lambda)};t^{-m_x(\mu)} \\ 
t^{1+\mu_{x+1}'-\lambda_x'} \end{matrix} 
; t,t \right) = (t^{1+\mu_x'-\lambda_x'};t)_{m_x(\lambda)} t^{-m_x(\lambda)m_x(\mu)},
\end{equation}
which shows \eqref{eq:qhyp_convergence}. 

We now show \eqref{eq:qhyp_convergence_degenerate}, so let $x=0$. Then $\mu_x(n+J)' = n+J$, so the arguments of the $q$-hypergeometric function in \eqref{eq:qhyp_convergence} are independent of $J$ except for $t^{-m_0(\mu(n+J))}$. In the sum
\begin{multline*}
t^{m_x(\lambda)m_x(\mu(n+J))}\qhypm
     \left(\begin{matrix} 
t^{-m_x(\lambda)};t^{-m_x(\mu(n+J))},0 \\ 
t^{1+\mu(n+J)_{x+1}'-\lambda_x'},t^{1+J-\mu(n+J)_x'+\lambda_{x+1}'}\end{matrix} 
; t,t \right) 
= t^{m_x(\lambda)m_x(\mu(n+J))} \\
\times \sum_{k=0}^{m_x(\lambda)} t^k \frac{(t^{-m_x(\lambda)};t)_k}{(t;t)_k}(t^{-m_x(\mu(n+J))};t)_k (t^{k+1+\mu(n+J)_{x+1}'-\lambda_x'};t)_{m_x(\lambda)-k} (t^{k+1-n+\lambda_{x+1}'};t)_{m_x(\lambda)-k},
\end{multline*}
the dominant term as $J \to \infty$ is the $k=m_x(\lambda)$ term, and its limit when normalized by $t^{m_x(\lambda)m_x(\mu(n+J))}$ is $1$. This shows \eqref{eq:qhyp_convergence_degenerate}.

The proof of \eqref{eq:princ_skew_Q} using \eqref{eq:finite_skew_Q} is exactly analogous except that only \eqref{eq:qhyp_convergence} is needed because there are only $n$ paths.
\end{proof}

\section{The {$t$}-deformed Gelfand-Tsetlin graph and its boundary} \label{sec:branching_graphs}

Let $t \in (0,1)$ for the remainder of the section. In this section we introduce the Hall-Littlewood Gelfand-Tsetlin graph and the notion of its boundary, the set of extreme coherent systems. The main result stated earlier, \Cref{thm:boundary}, is that the boundary is naturally in bijection with the set $\Sig_\infty$ of infinite signatures. We will break it into three parts: \Cref{thm:find_boundary} gives an explicit coherent system of measures $(M_n^\mu)_{n \geq 1}$ for each $\mu \in \Sig_\infty$, \Cref{thm:all_boundary1} tells that every extreme coherent system must be one of these, and \Cref{thm:all_boundary2} tells that each system $(M_n^\mu)_{n \geq 1}$ is extreme. 

The general structure of the proof of \Cref{thm:boundary}, via the so-called Vershik-Kerov ergodic method, is similar to e.g. \cite[Theorem 6.2]{olshanski2016extended} or \cite{cuenca2018asymptotic}. A good general reference for (unweighted) graded graphs, with references to research articles, is \cite[Chapter 7]{borodin2017representations}. 

\subsection{Classifying the boundary.}

\begin{defi}\label{def:our_graph}
$\G$ is the weighted, graded graph with vertices
\[
\bigsqcup_{n \geq 1} \Sig_n
\]
partitioned into \emph{levels} indexed by $\Z_{\geq 1}$. The only edges of $\G$ are between vertices on levels differing by $1$.
Between every $\lambda \in \Sig_n, \mu \in \Sig_{n+1}$ there is a weighted edge with weight
\[
L_n^{n+1}(\mu,\lambda) := P_{\mu/\lambda}(t^n) \frac{P_\lambda(1,\ldots,t^{n-1})}{P_\mu(1,\ldots,t^n)},
\]
and these weights are called \emph{cotransition probabilities} or \emph{(stochastic) links}. We use $L^{n+1}_n$ to denote the (infinite) $\Sig_{n+1} \times \Sig_n$ matrix with these weights.
\end{defi}

Note $L^{n+1}_n$ is a stochastic matrix by the branching rule. More generally, for $m \in \Z_{\geq 1} \cup \{\infty\}$, $1 \leq n < m$, and $\mu \in \Sig_m, \lambda \in \Sig_n$ we let
\begin{equation}\label{eq:def_Lmn}
L^m_n(\mu,\lambda) := P_{\mu/\lambda}(t^n,\ldots,t^{m-1}) \frac{P_\lambda(1,\ldots,t^{n-1})}{P_\mu(1,\ldots,t^{m-1})}.    
\end{equation}
When $m$ is finite one has $L^m_n = L^{n+1}_n L^{n+2}_n \cdots L^m_{m-1}$, where the product is just the usual matrix product. 

\begin{rmk}
The cotransition probabilities define (deterministic) maps $\cM(\Sig_m) \to \cM(\Sig_n)$, where here and below we use $\cM$ to denote the space of Borel probability measures, in this case with respect to the discrete topology on the set of signatures.
\end{rmk}

\begin{rmk}\label{rmk:trans_invariant_link}
\Cref{rmk:P_trans_inv} implies translation-invariance 
\begin{equation}\label{eq:trans_invariant_link}
  L^m_n(\mu,\lambda) = L^m_n(\mu+D[m],\lambda+D[n])  
\end{equation}
of the cotransition probabilities.
\end{rmk}

The cotransition probabilities have explicit formulas courtesy of the results of \Cref{sec:skew_formulas}, which will be useful in the proofs of \Cref{thm:conv_means_regular} and \Cref{thm:all_boundary1} later. For $\lambda \in \Sig_n$, we let 
\[
\sqbinom{n}{\lambda}_t = \frac{(t;t)_n}{\prod_{i \in \Z} (t;t)_{m_i(\lambda)}}
\]
(the $t$-deformed multinomial coefficient).

\begin{cor}\label{thm:link_explicit}
For $\mu \in \Sig_{n+J}, \lambda \in \Sig_n$,
\begin{equation}\label{eq:link_explicit}
    L^{n+J}_n(\mu,\lambda) = \frac{1}{\sqbinom{n+J}{J}_t} \prod_{x \in \Z} \frac{t^{(n-\lambda_{x}')(\mu_{x}'-\lambda_{x}') + m_x(\lambda)m_x(\mu)}}{(t;t)_{m_x(\lambda)}}\qhypm
     \left(\begin{matrix} 
t^{-m_x(\lambda)};t^{-m_x(\mu)},0 \\ 
t^{1+\mu_{x+1}'-\lambda_x'},t^{1+J-\mu_x'+\lambda_{x+1}'}\end{matrix} 
; t,t \right).
\end{equation}
\end{cor}
\begin{proof}
By the translation-invariance of \Cref{rmk:trans_invariant_link}, it suffices to prove the case when $\mu,\lambda$ are nonnegative signatures. We combine the formula of \Cref{thm:finite_skew_fn_computation} for $P_{\mu/\lambda}(t^n,\ldots,t^{n+J-1})$ with the one from \Cref{thm:hl_principal_formulas} for the principally specialized non-skew Hall-Littlewood polynomial. By the latter,
\begin{equation}\label{eq:other_factor}
    \frac{P_{\lambda}(1,\ldots,t^{n-1})}{P_{\mu}(1,\ldots,t^{n+J-1})} = \frac{(t;t)_n}{(t;t)_{n+J}} t^{n(\lambda)-n(\mu)} \prod_{i \geq 0} \frac{(t;t)_{m_i(\mu)}}{(t;t)_{m_i(\lambda)}}.
\end{equation}
Note also that by the definition of $n(\lambda)$,
\begin{equation}\label{eq:convert_from_n}
    t^{n(\lambda)-n(\mu)} = \prod_{x \geq 0} t^{\binom{\lambda_{x+1}'}{2}-\binom{\mu_{x+1}'}{2}},
\end{equation}
so by the identity 
\[
\binom{a+b}{2}-\binom{a}{2}-\binom{b}{2} = ab
\]
we have
\begin{equation}\label{eq:simplify_t_exp}
    t^{n(\lambda)-n(\mu)}\prod_{x \geq 0} t^{\binom{\mu_{x+1}'-\lambda_{x+1}'}{2}} = \prod_{x \geq 0} t^{-\lambda_{x+1}'(\mu_{x+1}'-\lambda_{x+1}')}.
\end{equation}

Simplifying the product of \eqref{eq:finite_skew_formula} with \eqref{eq:other_factor} by the above manipulations yields
\begin{align*}
\begin{split}
L^{n+J}_n(\mu,\lambda) = \frac{1}{\sqbinom{n+J}{J}_t} \prod_{x \geq 0} \frac{t^{(n-\lambda_{x+1}')(\mu_{x+1}'-\lambda_{x+1}') + m_x(\lambda)m_x(\mu)}}{(t;t)_{m_x(\lambda)}}\qhypm
     \left(\begin{matrix} 
t^{-m_x(\lambda)};t^{-m_x(\mu)},0 \\ 
t^{1+\mu_{x+1}'-\lambda_x'},t^{1+J-\mu_x'+\lambda_{x+1}'}\end{matrix} 
; t,t \right).
\end{split}
\end{align*}
The product may be extended to all $x \in \Z$ since all other terms are $1$, at which point it is manifestly translation-invariant, which yields the result for arbitrary signatures.
\end{proof}

\begin{defi}
A sequence $(M_n)_{n \geq 1}$ of probability measures on $\Sig_1,\Sig_2,\ldots$ is \emph{coherent} if 
\[
\sum_{\mu \in \Sig_{n+1}} M_{n+1}(\mu) L_n^{n+1}(\mu,\lambda) = M_n(\lambda)
\]
for each $n \geq 1$ and $\lambda \in \Sig_n$.
\end{defi}

\begin{defi}
A coherent system of measures $(M_n)_{n \geq 1}$ is \emph{extreme} if there do not exist coherent systems $(M_n')_{n \geq 1},(M_n'')_{n \geq 1}$ different from $(M_n)_{n \geq 1}$ and $s \in (0,1)$ such that $M_n = sM_n' + (1-s)M_n''$ for each $n$. The set of extreme coherent systems of measures on a weighted, graded graph is called its \emph{boundary}, and denoted in our case by $\partial \G$.
\end{defi}

In the previous section we considered both signatures (of finite length), and integer partitions, which have infinite length but stabilize to $0$. To describe points on the boundary $\partial \G$ in this section, it turns out that it will be necessary to introduce signatures of infinite length which are not partitions.

\begin{defi}
We denote the set of infinite signatures by
\[
\Sig_\infty := \{(\mu_1,\mu_2,\ldots) \in \Z^\infty: \mu_1 \geq \mu_2 \geq \ldots\}.
\]
We refer to the $\mu_i$ as \emph{parts} just as with partitions, and define $\mu_i'$ and $m_i(\mu)$ the exact same way, though we must allow them to be equal to $\infty$.
\end{defi}

A distinguished subset of $\Sig_\infty$ is $\Y$, the set of partitions. Translating by any $D \in \Z$ yields 
\[
\Y + D = \{ \mu \in \Sig_\infty: \mu_i = D \text{ for all but finitely many }i\}.
\]
However, $\Sig_\infty$ also contains infinite signatures with parts not bounded below, the set of which we denote by 
\[
\Sig_\infty^{unstable} := \{\mu \in \Sig_\infty: \lim_{i \to \infty} \mu_i = -\infty\}.
\]
It is clear that 
\[
\Sig_\infty = \Sig_\infty^{unstable} \sqcup \bigsqcup_{D \in \Z} (\Y + D)
\]
and we will use this decomposition repeatedly in what follows. To treat the unbounded signatures we will approximate by signatures in $\Y+D$, which are no more complicated than partitions, and to this end we introduce the following notation.

\begin{defi}
For $\lambda \in \Sig_\infty^{unstable}$ and $D \in \Z$, we let 
\[
\lambda^{(D)} = (\lambda_1,\ldots,\lambda_k, D, D, \ldots) \in \Y+D
\]
where $k$ is the largest index such that $\lambda_k > D$.
\end{defi}

The first step to proving \Cref{thm:boundary} is, for each element of $\Sig_\infty$, an explicit formula for a corresponding coherent system of measures on $\G$; we will later show that these are exactly the boundary points. 

\begin{prop}\label{thm:find_boundary}
For each $\mu \in \Sig_\infty$, there exists a coherent system of measures $(M_n^{\mu})_{n \geq 1}$ on $\G$, given explicitly by
\begin{equation}\label{eq:explicit_measures}
    M_n^\mu(\lambda) := 
    \sqbinom{n}{\lambda}_t \prod_{x \in \Z} t^{(\mu'_x - \lambda'_x)(n-\lambda'_x)} (t^{1+\mu'_x-\lambda'_x};t)_{m_x(\lambda)}.
\end{equation}
for $\lambda \in \Sig_n$.
\end{prop}

Before proving \Cref{thm:find_boundary} we will calculate explicit formulas for the links $L^m_n$ in \Cref{thm:finite_skew_fn_computation}, which are a corollary to the formula for principally specialized skew functions in \Cref{thm:skew_formula}. 

\begin{cor}\label{thm:measure_and_monotone}
Let $n \geq 1$. If $ \lambda \in \Sig_n^{\geq 0}, \mu \in \Y$, then
\begin{equation}\label{eq:stable_formula}
    \frac{P_{\mu/\lambda}(t^n,t^{n+1},\ldots)P_{\lambda}(1,\ldots,t^{n-1})}{P_{\mu}(1,t,\ldots)} = \sqbinom{n}{\lambda}_t \prod_{x \in \Z_{>0}} t^{(\mu'_x - \lambda'_x)(n-\lambda'_x)} (t^{1+\mu'_x-\lambda'_x};t)_{m_x(\lambda)}.
\end{equation}
Furthermore, if instead $\lambda \in \Sig_n, \mu \in \Sig_\infty^{unstable}$, then
\begin{equation}\label{eq:unstable_formula}
    \frac{P_{(\mu^{(D)} - D[\infty])/(\lambda - D[n])}(t^n,t^{n+1},\ldots)P_{(\lambda - D[n])}(1,\ldots,t^{n-1})}{P_{(\mu^{(D)}-D[\infty])}(1,t,\ldots)} 
\end{equation}
increases monotonically as $D \to -\infty$, and stabilizes to 
\begin{equation}
    \sqbinom{n}{\lambda}_t \prod_{x \in \Z} t^{(\mu'_x - \lambda'_x)(n-\lambda'_x)} (t^{1+\mu'_x-\lambda'_x};t)_{m_x(\lambda)}
\end{equation}
for all $D < \lambda_n$.
\end{cor}
\begin{proof}
\eqref{eq:stable_formula} follows from \Cref{thm:skew_formula} and \Cref{thm:hl_principal_formulas} by the same proof as that of \Cref{thm:link_explicit}, so let us show the monotonicity and stabilization statement. Substituting \eqref{eq:unstable_formula} into \eqref{eq:stable_formula} and changing variables $x \mapsto x+D$ in the product yields
\begin{multline*}
    \frac{P_{(\mu^{(D)} - D[\infty])/(\lambda - D[n])}(t^n,t^{n+1},\ldots)P_{(\lambda - D[n])}(1,\ldots,t^{n-1})}{P_{(\mu^{(D)}-D[\infty])}(1,t,\ldots)}  = \sqbinom{n}{\lambda}_t \prod_{x \in \Z_{>D}} t^{(\mu'_x - \lambda'_x)(n-\lambda'_x)} (t^{1+\mu'_x-\lambda'_x};t)_{m_x(\lambda)}.
\end{multline*}
The factors in the product are all in $[0,1]$ and are equal to $1$ when $x \leq \lambda_n$, and since the product is over $x \in \Z_{>D}$ this completes the proof.
\end{proof}

\begin{rmk}\label{rmk:boundary_links_infty}
Given the translation-invariance of the links $L^m_n$ noted in \Cref{rmk:trans_invariant_link}, when $\mu \in \Y+D$ it is natural to view the expression 
\[
\frac{P_{(\mu - D[\infty])/(\lambda - D[n])}(t^n,t^{n+1},\ldots)P_{(\lambda - D[n])}(1,\ldots,t^{n-1})}{P_{(\mu-D[\infty])}(1,t,\ldots)}
\]
as simply 
\[
\frac{P_{\mu/\lambda}(t^n,t^{n+1},\ldots)P_{\lambda}(1,\ldots,t^{n-1})}{P_{\mu}(1,t,\ldots)},
\]
even though in our setup the expressions $P_{\mu/\lambda}(t^n,t^{n+1},\ldots)$ and $P_{\mu}(1,t,\ldots)$ are not well-defined when $\mu$ is not in $\Y$. Hence in view of \Cref{thm:skew_formula} it is natural to view the coherent systems $(M^\mu_n)_{n \geq 1}$ of \Cref{thm:boundary} as being given by links `at infinity'
\[
M^\mu_n(\lambda) \text{``}=\text{''} L^\infty_n(\mu,\lambda) = \frac{P_{\mu/\lambda}(t^n,t^{n+1},\ldots)P_{\lambda}(1,\ldots,t^{n-1})}{P_{\mu}(1,t,\ldots)}
\]
for general $\mu \in \Sig_\infty$, though we must take a slightly roundabout path to make rigorous sense of the RHS. Many of the proofs below follow the same pattern of proving a result for $\mu \in \Y$ by usual symmetric functions machinery, appealing to translation-invariance for $\mu \in \Y+D$, and then approximating $\mu \in \Sig_\infty^{unstable}$ by elements $\mu^{(D)} \in \Y+D$ and using \Cref{thm:measure_and_monotone} to apply the monotone convergence theorem. We note also that the formula \eqref{eq:explicit_measures} is clearly translation-invariant.
\end{rmk}

\begin{proof}[Proof of {\Cref{thm:find_boundary}}]
We first show $M_n^\mu$ is indeed a probability measure. Clearly it is a nonnegative function on $\Sig_n$, but we must show it sums to $1$. When $\mu \in \Y$ this is by \Cref{thm:measure_and_monotone} and the definition of skew HL polynomials, and the case $\mu \in \Y+D$ reduces to this one. Hence it remains to show that for $\mu \in \Sig_\infty^{unstable}$,
\begin{equation}\label{eq:is_prob_measure}
    \sum_{\lambda \in \Sig_n} \lim_{D \to -\infty} \frac{P_{(\mu^{(D)} - D[\infty])/(\lambda - D[n])}(t^n,t^{n+1},\ldots)P_{(\lambda - D[n])}(1,\ldots,t^{n-1})}{P_{(\mu^{(D)}-D[\infty])}(1,t,\ldots)} = 1.
\end{equation}
By \Cref{thm:measure_and_monotone}, the functions
\begin{multline*}
\frac{P_{(\mu^{(D)} - D[\infty])/(\lambda - D[n])}(t^n,t^{n+1},\ldots)P_{(\lambda - D[n])}(1,\ldots,t^{n-1})}{P_{(\mu^{(D)}-D[\infty])}(1,t,\ldots)}  = \sqbinom{n}{\lambda}_t \prod_{x \in \Z_{>D} } t^{(\mu'_x - \lambda'_x)(n-\lambda'_x)} (t^{1+\mu'_x-\lambda'_x};t)_{m_x(\lambda)}
\end{multline*}
converge to the summand in \eqref{eq:is_prob_measure} from below as $D \to -\infty$. Hence \eqref{eq:is_prob_measure} follows by the monotone convergence theorem.

For $\mu \in \Y+D$ for some $D$, coherency again follows from the definition of skew functions and the first part of \Cref{thm:measure_and_monotone}. For $\mu \in \Sig_\infty^{unstable}$ we must show 
\begin{multline}\label{eq:coherency_wts}
    \sum_{\kappa \in \Sig_{n+1}} \lim_{D \to -\infty} \frac{P_{(\mu^{(D)} - D[\infty])/(\kappa - D[n+1])}(t^{n+1},\ldots)P_{(\kappa - D[n+1])}(1,\ldots,t^{n})}{P_{(\mu^{(D)}-D[\infty])}(1,t,\ldots)} \frac{P_{\kappa/\lambda}(t^n)P_\lambda(1,\ldots,t^{n-1})}{P_\kappa(1,\ldots,t^n)} \\
    = \lim_{D \to -\infty} \frac{P_{(\mu^{(D)} - D[\infty])/(\lambda - D[n])}(t^n,t^{n+1},\ldots)P_{(\lambda - D[n])}(1,\ldots,t^{n-1})}{P_{(\mu^{(D)}-D[\infty])}(1,t,\ldots)}.
\end{multline}
Again the monotone convergence theorem allows us to interchange the limit and sum. The result then follows by translation invariance of the links \eqref{eq:trans_invariant_link} and the definition of skew HL polynomials.
\end{proof}

It remains to show that the coherent systems identified in \Cref{thm:find_boundary} are extreme and that all extreme coherent systems are of this form. Just from the definition, an arbitrary extreme coherent system is an elusive object. Luckily, the general results of the Vershik-Kerov ergodic method guarantee that extreme coherent systems can be obtained through limits of cotransition probabilities for certain \emph{regular sequences} of signatures, which are much more concrete.

\begin{defi}\label{def:regular}
A sequence $(\mu(n))_{n \geq 1}$ with $\mu(n) \in \Sig_n$ is \emph{regular} if for every $k \in \Z_{\geq 1}$ and $ \lambda \in \Sig_k$, the limit 
\[
M_k(\lambda) := \lim_{n \to \infty} L^n_k(\mu(n),\lambda)
\]
exists and $M_k$ is a probability measure.
\end{defi}

\begin{prop}\label{thm:approximation}
For any extreme coherent system $(M_k)_{k \geq 1} \in \partial G$ there exists a regular sequence $(\mu(n))_{n \geq 1}$ such that 
\[
M_k(\cdot) = \lim_{n \to \infty} L_k^n(\mu(n),\cdot).
\]
\end{prop}
\begin{proof}
Follows from {\cite[Theorem 6.1]{okounkov1998asymptotics}}.
\end{proof}

The space of extreme coherent systems obtained from regular sequences as in \Cref{thm:approximation} is sometimes referred to as the \emph{Martin boundary}. It naturally includes into the boundary, and \Cref{thm:approximation} says that in this setup they are in fact equal. 

\begin{lem}\label{thm:conv_means_regular}
Let $(\mu(n))_{n \geq 1}$ be a sequence with $\mu(n) \in \Sig_n$, such that 
\[
\lim_{n \to \infty} \mu(n)_i =: \mu_i
\]
exists and is finite for every $i$. Then $(\mu(n))_{n \geq 1}$ is regular and the corresponding coherent family is $(M_n^\mu)_{n \geq 1}$, where $\mu = (\mu_1,\mu_2,\ldots) \in \Sig_\infty$.
\end{lem}
\begin{proof}
Let $(\mu(n))_{n \geq 1}$ satisfy the hypothesis. We must show for arbitrary $k,\lambda \in \Sig_k$ that 
\begin{equation}\label{eq:reg_wts}
    \lim_{n \to \infty} L_k^n(\mu(n),\lambda) = M^\mu_k(\lambda).
\end{equation}
It is easy to see from the explicit formula in \Cref{thm:link_explicit} that $L_k^n(\mu(n),\lambda)$ depends only on the parts of $\mu(n)$ which are $\geq \lambda_k$. For any fixed $x$, it is easy to see that $\mu(n)_x' \to \mu_x'$. In fact, for all sufficiently large $n$, it must be true that $\mu(n)_x' = \mu_x'$ for all $x \geq \lambda_k$ such that $\mu_x'$ is finite. Hence for all sufficiently large $n$ the product in \eqref{eq:link_explicit} only has nontrivial terms when $\lambda_k \leq x \leq \mu_1$, so it suffices to show that each term converges. This follows by the exact same argument as the proof of \Cref{thm:skew_formula}, again with two cases based on whether $\mu(n)_x'$ stabilizes or $\mu(n)_x' \to \infty$. 
\end{proof}

\begin{lem}\label{thm:measure_vanishing}
Let $1 \leq k < n$ be integers and $\lambda \in \Sig_k$.
\begin{enumerate}
    \item If $\mu \in \Sig_n$ is such that $\lambda'_x > \mu'_x$ for some $x$, then $L_k^n(\mu,\lambda)=0$.
    \item If $\mu \in \Sig_\infty$ is such that $\lambda'_x > \mu'_x$ for some $x$, then $M_k^\mu(\lambda) = 0$.
\end{enumerate}
\end{lem}
\begin{proof}
If $\mu \in \Y,\lambda \in \Y$ and $\lambda'_x > \mu'_x$ for any $x$, it follows from the upper-triangularity of the branching rule \cite[Chapter III, (5.5')]{mac} that $P_{\lambda/\mu}(t^k,\ldots,t^{n-1}) = 0$, showing (1). Approximating $\mu \in \Sig_\infty$ with $(\mu_1,\ldots,\mu_n) \in \Sig_n$ and invoking \Cref{thm:conv_means_regular} yields (2).
\end{proof}

\Cref{thm:measure_vanishing} could also be shown by the explicit formula \eqref{eq:explicit_measures}, but as the above proof shows, it in fact requires only the very basic properties of symmetric functions. 

\begin{prop}\label{thm:all_boundary1}
Every extreme coherent system is given by $(M^\mu_n)_{n \geq 1}$ for some $\mu \in \Sig_\infty$.
\end{prop}
\begin{proof}
Let $(M_n)_{n \geq 1}$ be an extreme coherent system and $(\mu(n))_{n \geq 1}$ be a regular sequence converging to it, the existence of which is guaranteed by \Cref{thm:approximation}. We wish to find $\mu \in \Sig_\infty$ such that 
\begin{equation}\label{eq:wts_allboundary1}
    \lim_{n \to \infty} L^n_k(\mu(n),\lambda) = M^\mu_k(\lambda)
\end{equation}
for all $k$ and $\lambda \in \Sig_k$, and will construct $\mu$ as a limit of the signatures $\mu(n)$.

Our first step is to show the sequence of first parts $(\mu_1(n))_{n \geq 1}$ is bounded above (and hence all other $(\mu_i(n))_{n \geq 1}$ are as well). Suppose for the sake of contradiction that this is not the case. Then there is a subsequence $(\mu_1(n_j))_{j \geq 1}$ of $(\mu_1(n))_{n \geq 1}$ for which $\mu_1(n_j) \to \infty$. We claim that for any $k$ and $\lambda \in \Sig_k$, 
\begin{equation}\label{eq:tightness_fails}
    \lim_{j \to \infty} L^{n_j}_k(\mu^(n_j),\lambda) = 0.
\end{equation}
This suffices for the contradiction, as then \eqref{eq:tightness_fails} holds also with $n_j$ replaced by $n$ by regularity of $(\mu(n))_{n \geq 1}$, therefore the sequence of probability measures $L^{n}_k(\mu(n),\cdot)$ converges to the zero measure, which contradicts the definition of regular sequence. So let us prove \eqref{eq:tightness_fails}, and to declutter notation let us without loss of generality denote the subsequence by $(\mu(n))_{n \geq 1}$ as well.

We claim there exists a constant $C_k$ such that for all $J \geq 1$ and $\nu \in \Sig_{k+J}$, 
\begin{equation}\label{eq:qhyp_to_bound}
\left| t^{m_x(\nu)m_x(\lambda)}\qhypm
     \left(\begin{matrix} 
t^{-m_x(\lambda)};t^{-m_x(\nu)},0 \\ 
t^{1+\nu_{x+1}'-\lambda_x'},t^{1+J-\nu_x'+\lambda_{x+1}'}\end{matrix} 
; t,t \right) \right| \leq C_k
\end{equation}
For fixed $\lambda$, $1+\nu_{x+1}'-\lambda_x'$ and $1+J-\nu_x'+\lambda_{x+1}'$ are both bounded below independent of $\nu$ by $1-k$. This gives an upper bound on the factors $(bt^\ell;t)_{m_x(\lambda)-\ell}, 0 \leq \ell \leq m_x(\lambda)$ where $b \in \{t^{1+\nu_{x+1}'-\lambda_x'},t^{1+J-\nu_x'+\lambda_{x+1}'}\}$ which appear in the sum expansion \eqref{eq:qhyp_def} of \eqref{eq:qhyp_to_bound}. The term $t^{m_x(\nu)m_x(\lambda)}(t^{-m_x(\nu)};t)_{\ell}$ is likewise bounded above independent of $\nu$. Because $m_x(\lambda)$ and $\lambda_x'$ can only take finitely many values, the claim follows. Furthermore, the LHS of \eqref{eq:qhyp_to_bound} is simply $1$ whenever $m_x(\lambda)=0$, which is true for all but finitely many $x$. Plugging this bound into \Cref{thm:link_explicit} yields 
\begin{equation}\label{eq:easy_to_bound}
    L^n_k(\mu(n),\lambda) \leq \frac{C_k^k}{\sqbinom{n}{k} \prod_{i \in \Z} (t;t)_{m_i(\lambda)}} \prod_{x \in \Z} t^{(k-\lambda_x')(\mu(n)_x'-\lambda_x')}.
\end{equation}
For $\lambda_1 < x \leq \mu(n)_1$, one has $t^{(k-\lambda_x')(\mu(n)_x'-\lambda_x')} \leq t^k < 1$, and our claim \eqref{eq:tightness_fails} follows.

Now, suppose for the sake of contradiction that there exists $k$ for which $(\mu(n)_k)_{n \geq 1}$ is not bounded below. Then for any $\lambda \in \Sig_k$, there are infinitely many $n$ for which $\mu(n)_k < \lambda_k$ and consequently $\mu(n)_x' < \lambda_x'=k$ for $x=\lambda_k$. By \Cref{thm:measure_vanishing}, $L^n_k(\mu(n),\lambda)=0$ for all such $n$, therefore $L^n_k(\mu(n),\lambda) \to 0$ as $n \to \infty$ since $(\mu(n))_{n \geq 1}$ is a regular sequence. This is a contradiction, therefore $(\mu(n)_k)_{n \geq 1}$ is bounded below for each $k$. 

Since $(\mu(n)_k)_{n \geq 1}$ is bounded above and below for each $k$, there is a subsequence on which these converge, and by a diagonalization argument there exists a subsequence $(\mu(n_j))_{j \geq 1}$ on which $\mu(n_j)_k$ converges for every $k$. Letting $\mu_i = \lim_{j \to \infty} \mu(n_j)_i$ and $\mu = (\mu_1,\mu_2,\ldots) \in \Sig_\infty$, we have by \Cref{thm:conv_means_regular} that 
\begin{equation*}
    \lim_{j \to \infty} L_k^{n_j}(\mu(n_j),\lambda) = M^\mu_k(\lambda)
\end{equation*}
for each $\lambda \in \Sig_k$. Since $\lim_{n \to \infty} L_k^n(\mu(n),\lambda)$ exists by the definition of regular sequence, it must also be equal to $M^\mu_k(\lambda)$. This shows \eqref{eq:wts_allboundary1}, completing the proof.
\end{proof}

For the other direction, \Cref{thm:all_boundary2}, we will need the basic fact that general coherent systems are convex combinations of extreme ones. 

\begin{prop}\label{thm:extreme_combinations}
For any coherent system $(M_n)_{n \geq 1}$ on $\G$, there exists a Borel\footnote{The topology on $\partial \G$ here is the following. For each $n$, the set of probability measures on $\Sig_n$ inherits a topology from the product topology on $\R^\infty$ by viewing the measures as functions, which gives a topology on the inverse limit $\partial \G$.} measure $\pi$ on $\partial \G$ such that 
\[
M_k = \int_{M' \in \partial \G} M_k' \pi(dM')
\]
for each $k$, where $M'$ is shorthand for a coherent system $(M'_n)_{n \geq 1}$.
\end{prop}
\begin{proof}
Follows from \cite[Theorem 9.2]{olshanski2003problem}.
\end{proof}

It will also be necessary to put a topology on $\Sig_\infty$, namely the one inherited from the product topology on $\Z^\infty$ where $\Z$ is equipped with the cofinite topology. The following lemma shows that these natural choices of topology on $\Sig_\infty$ and $\partial \G$ are compatible.

\begin{lem}\label{thm:check_topologies}
The map 
\begin{align*}
    f:\Sig_\infty &\to \cM(\partial \G) \\
     \mu &\mapsto (M^\mu_n)_{n \geq 1}
\end{align*}
is continuous, hence in particular Borel.
\end{lem}
\begin{proof}
Since $\Sig_\infty$ is first-countable, to show $f$ is continuous it suffices to show it preserves limits of sequences. Hence we must show that for any $\mu \in \Sig_\infty$, if $\nu^{(1)},\nu^{(2)},\ldots \in \Sig_\infty$ and $\nu^{(k)}_i \to \mu_i$ for all $i$, then $M^{\nu^{(k)}}_n \to M^\mu_n$ pointwise as functions on $\Sig_n$. This follows straightforwardly from the explicit formula \eqref{eq:explicit_measures} of \Cref{thm:find_boundary}.
\end{proof}

\begin{prop}\label{thm:all_boundary2}
For every $\mu \in \Sig_\infty$, the coherent system $(M^\mu_n)_{n \geq 1}$ is extreme.
\end{prop}
\begin{proof}
Fix $\mu \in \Sig_\infty$. By \Cref{thm:all_boundary1},
there is a Borel measure $\pi \in \cM(\partial \G)$.
\begin{equation}\label{eq:combination_pushforward}
    M^\mu_k = \int_{M' \in \partial \G} M_k' \pi(dM') = \int_{\nu \in \Sig_\infty} M^\nu_k (\iota_* \pi)(d\nu)
\end{equation}
where $\iota: \partial \G \inj \Sig_\infty$ is the inclusion guaranteed by \Cref{thm:all_boundary1}. Because $f \circ \iota = \Id$ and $f$ is Borel, $\iota$ is a Borel isomorphism onto its image, hence $\iota_* \pi$ is a Borel measure in the topology on $\Sig_\infty$ above.

We first claim that $\iota_* \pi$ is supported on 
\[
S_{\leq \mu} := \{\nu \in \Sig_\infty: \nu_i \leq \mu_i \text{ for all }i\}. 
\]
Suppose not. Since 
\[
\Sig_\infty \setminus S_{\leq \mu} = \bigcup_{k \geq 1} \{\nu \in \Sig_\infty: \nu_i > \mu_i \text{ for at least one }1 \leq i \leq k\}
\]
and 
\[
\{\nu \in \Sig_\infty: \nu_i > \mu_i \text{ for at least one }1 \leq i \leq k\} = \bigcup_{\substack{\lambda \in \Sig_k: \\ \exists i \text{ s.t. } \lambda_i > \mu_i}} \{\nu \in \Sig_\infty: \nu_i = \lambda_i \text{ for all }1 \leq i \leq k\},
\]
if $(\iota_* \pi)(\Sig_\infty \setminus S_{\leq \mu}) > 0$ then there exists $k$ and $\lambda \in \Sig_k$ such that 
\begin{equation}\label{eq:bad_set}
  (\iota_* \pi)(\{\nu \in \Sig_\infty: \nu_i = \lambda_i \text{ for all }1 \leq i \leq k\}) > 0.  
\end{equation}
Denoting the set in \eqref{eq:bad_set} by $S_k(\lambda) \subset \Sig_\infty$, we have
\begin{equation}\label{eq:wrong_bound}
    M_k^\mu(\lambda_1,\ldots,\lambda_k) = \int_{\nu \in S_k(\lambda)} M_k^\nu(\lambda_1,\ldots,\lambda_k) (\iota_* \pi)(d\nu) + \int_{\nu \in \Sig_\infty \setminus S_k(\lambda)} M_k^\nu(\lambda_1,\ldots,\lambda_k) (\iota_* \pi)(d\nu).
\end{equation}
The LHS is $0$ by \Cref{thm:measure_vanishing}. If $\nu \in S_k(\lambda)$, then the only factor in 
\[
M_k^\nu(\lambda_1,\ldots,\lambda_k) = \sqbinom{k}{\lambda}_t \prod_{x \in \Z_{\geq \lambda_k}} t^{(\nu'_x - \lambda'_x)(k-\lambda'_x)} (t^{1+\nu'_x-\lambda'_x};t)_{m_x(\lambda)}
\]
which depends on $\nu$ is $(t^{1+\nu'_{\lambda_k}-k};t)_{m_{\lambda_k}(\lambda)}$, which is clearly bounded below by $(t;t)_\infty$. Hence the RHS of \eqref{eq:wrong_bound} is bounded below by 
\[
(\iota_* \pi)(S_k(\lambda)) (t;t)_\infty  \sqbinom{k}{\lambda}_t > 0,
\]
a contradiction. Therefore $\iota_* \pi$ is indeed supported on $S_{\leq \mu}$. 

For each $k \geq 1$ we may decompose 
\[
S_{\leq \mu} = \left(S_{\leq \mu} \cap  S_k(\mu_1,\ldots,\mu_k)\right) \sqcup \left( S_{\leq \mu} \cap \left( S_k(\mu_1,\ldots,\mu_k)\right)^c\right)
\]
into those signatures which agree with $\mu$ on the first $k$ coordinates and those which do not, and
\begin{multline}\label{eq:less_or_equal_ints}
    M^\mu_k(\mu_1,\ldots,\mu_k) = \int_{\nu \in S_{\leq \mu} \cap S_k(\mu_1,\ldots,\mu_k)} M^\nu_k(\mu_1,\ldots,\mu_k) (\iota_* \pi)(d\nu) \\ + \int_{\nu \in S_{\leq \mu} \cap (S_k(\mu_1,\ldots,\mu_k)^c)} M^\nu_k(\mu_1,\ldots,\mu_k) (\iota_* \pi)(d\nu).
\end{multline}
The second integral in \eqref{eq:less_or_equal_ints} is always $0$ by \Cref{thm:measure_vanishing}. If $\nu \in S_{\leq \mu} \cap S_k(\mu_1,\ldots,\mu_k)$ then $\nu_x'=\mu_x'$ for $x > \mu_k$ and $\nu_x' \leq \mu_x'$ when $x=\mu_k$. Hence 
\[
(t^{1+\nu_x'-k};t)_{m_x(\mu_1,\ldots,\mu_k)} \leq (t^{1+\mu_x'-k};t)_{m_x(\mu_1,\ldots,\mu_k)}
\]
for all $x$, and all other factors in \eqref{eq:explicit_measures} are the same for $M_k^\nu(\mu_1,\ldots,\mu_k)$ and $M_k^\mu(\mu_1,\ldots,\mu_k)$, therefore 
\[
M_k^\nu(\mu_1,\ldots,\mu_k) \leq M_k^\mu(\mu_1,\ldots,\mu_k) \text{ for all }\nu \in S_{\leq \mu} \cap S_k(\mu_1,\ldots,\mu_k).
\]
Hence \eqref{eq:less_or_equal_ints} reduces to 
\begin{equation}
    M^\mu_k(\mu_1,\ldots,\mu_k) \leq M^\mu_k(\mu_1,\ldots,\mu_k) \cdot (\iota_* \pi)(S_{\leq \mu} \cap S_k(\mu_1,\ldots,\mu_k)).
\end{equation}
Since $M^\mu_k(\mu_1,\ldots,\mu_k)  > 0$ by \eqref{eq:explicit_measures}, it follows that 
\[
(\iota_* \pi)(S_{\leq \mu} \cap S_k(\mu_1,\ldots,\mu_k)) = 1.
\]
Since this is true for all $k$ and $\bigcap_k \left(S_{\leq \mu} \cap S_k(\mu_1,\ldots,\mu_k)\right) = \{\mu\}$, it follows that $(\iota_* \pi)(\{\mu\})=1$, i.e. $\iota_* \pi$ is the delta mass at $\mu$. Hence $(M^\mu_n)_{n \geq 1}$ is an extreme coherent system, completing the proof.
\end{proof}

\section{Infinite {$p$}-adic random matrices and corners} \label{sec:infinite_matrices}

In this section, we turn to $p$-adic random matrix theory and prove \Cref{thm:recover_BQ} and \Cref{thm:recover_assiotis}. We will first give the basic setup of $p$-adic random matrices and the key result \Cref{thm:p-adic_corners} which relates the operations of removing rows and columns to Hall-Littlewood polynomials. In \Cref{subsec:boundary_aux} we prove auxiliary boundary results on a slightly more complicated branching graph which extends the one in the previous section, which are tailored to the random matrix corner situation. We then use these to deduce the result \Cref{thm:recover_BQ}, that extreme bi-invariant measures on $\Mat_{\infty \times \infty}(\Q_p)$ are parametrized by the set $\bSig_\infty$ defined in \Cref{def:extended_sigs} below, from the parametrization of the boundary of this augmented branching graph by $\bSig_\infty$ (\Cref{thm:double_boundary}). 

\subsection{$p$-adic background.}

The following basic background is more or less quoted from \cite{van2020limits} and is a condensed version of the exposition in \cite[\S 2]{evans2002elementary}, to which we refer any reader desiring a more detailed introduction to $p$-adic numbers. Fix a prime $p$. Any nonzero rational number $r \in \Q^\times$ may be written as $r=p^k (a/b)$ with $k \in \Z$ and $a,b$ coprime to $p$. Define $|\cdot|: \Q \to \R$ by setting $|r|_p = p^{-k}$ for $r$ as before, and $|0|_p=0$. Then $|\cdot|_p$ defines a norm on $\Q$ and $d_p(x,y) :=|x-y|_p$ defines a metric. We define the \emph{field of $p$-adic numbers} $\Q_p$ to be the completion of $\Q$ with respect to this metric, and the \emph{$p$-adic integers} $\Z_p$ to be the unit ball $\{x \in \Q_p : |x|_p \leq 1\}$. It is not hard to check that $\Z_p$ is a subring of $\Q_p$. We remark that $\Z_p$ may be alternatively defined as the inverse limit of the system $\ldots \to \Z/p^{n+1}\Z \to \Z/p^n \Z \to \cdots \to \Z/p\Z \to 0$, and that $\Z$ naturally includes into $\Z_p$. 

$\Q_p$ is noncompact but is equipped with a left- and right-invariant (additive) Haar measure; this measure is unique if we normalize so that the compact subgroup $\Z_p$ has measure $1$, and we denote it by $\M$. The restriction of this measure to $\Z_p$ is the unique Haar probability measure on $\Z_p$, and is explicitly characterized by the fact that its pushforward under any map $r_n:\Z_p \to \Z/p^n\Z$ is the uniform probability measure. For concreteness, it is often useful to view elements of $\Z_p$ as `power series in $p$' $a_0 + a_1 p + a_2 p^2 + \ldots$, with $a_i \in \{0,\ldots,p-1\}$; clearly these specify a coherent sequence of elements of $\Z/p^n\Z$ for each $n$. The Haar probability measure then has the alternate explicit description that each $a_i$ is iid uniformly random from $\{0,\ldots,p-1\}$. Additionally, $\Q_p$ is isomorphic to the ring of Laurent series in $p$, defined in exactly the same way.

$\GL_n(\Z_p) \times \GL_m(\Z_p)$ acts on $\Mat_{n \times m}(\Q_p)$ by left- and right multiplication. The orbits of this action are parametrized by signatures with possibly infinite parts, which we now define formally.

\begin{defi}\label{def:extended_sigs}
For $n \in \Z_{\geq 1}$, we let 
\[
\bSig_n := \{(\lambda_1,\ldots,\lambda_n) \in (\Z \cup \{-\infty\})^n: \lambda_1 \geq \ldots \geq \lambda_n\},
\]
where we take $-\infty < a$ for all $a \in \Z$, and refer to elements of $\bSig_n$ as \emph{extended signatures}. The definition of $\bSig_\infty$ is exactly analogous. For $0 \leq k \leq n $, we denote by $\bSig_n^{(k)} \subset \bSig_n$ the set of all extended signatures with exactly $k$ integer parts and the rest equal to $-\infty$. For $\lambda \in \bSig_n^{(k)}$, we denote by $\lambda^* \in \Sig_k$ the signature given by its integer parts.
\end{defi}

The parametrization, stated below, is often called Smith normal form.

\begin{prop}\label{thm:smith}
Let $n \leq m$. For any $A \in \Mat_{n \times m}(\Q_p)$, there exists a unique $\lambda \in \bSig_n$ for which there exist $U \in \GL_n(\Z_p), V \in \GL_m(\Z_p)$ such that $UAV = \diag_{n \times m}(p^{-\lambda_1},\ldots,p^{-\lambda_n})$, where we formally take $p^\infty = 0$.
\end{prop}

For any $\lambda \in \bSig_{\min(m,n)}$, it is useful to consider the probability measure on $\Mat_{n \times m}(\Q_p)$ with distribution $U \diag_{n \times m}(p^{-\lambda_1},\ldots,p^{-\lambda_n}) V$ where $U \in \GL_n(\Z_p),V \in \GL_m(\Z_p)$ are Haar-distributed. By uniqueness of the Haar measure, this measure is the unique one which is left- and right-invariant under multiplication by $\GL_n(\Z_p)$ and $\GL_m(\Z_p)$ respectively.

\begin{defi}
We denote the extended signature in \Cref{thm:smith} by $\SN(A) \in \bSig_{\min(m,n)}$, and denote this signature padded with infinitely many parts equal to $-\infty$ by $\ESN(A) \in \bSig_\infty$. We refer to the finite parts of either signature as the \emph{singular numbers} of $A$. 
\end{defi}

The reason for padding with $-\infty$ is to allow us to treat matrices of different sizes on equal footing, essentially viewing them as corners of a large matrix of low rank. 
 It is somewhat unwieldy but seemed to be the least awkward formalism for the problem at hand.

\begin{rmk}
We have defined singular numbers with the opposite sign convention as \cite{van2020limits} (though the same sign convention as \cite{assiotis2020infinite,bufetov2017ergodic}) to match with the branching graph notation of \Cref{sec:branching_graphs} and the latter references. 
\end{rmk}

\begin{prop}\label{thm:p-adic_corners}
Let $n, m \geq 1$ be integers, $\mu \in \bSig_\infty$ with $\len(\mu^*) \leq \min(m+1,n)$, let $A \in \Mat_{n \times (m+1)}(\Q_p)$ be distributed by the unique bi-invariant measure with singular numbers $\mu$, and let $t=1/p$. If $A' \in \Mat_{n \times m}$ is the first $m$ columns of $A$, then $\ESN(A')$ is a random element of $\bSig_\infty$ with
    \begin{equation}\label{eq:p-adic_corners}
        \Pr( \ESN(A') = \lambda) = \begin{cases}
     \frac{Q_{-\lambda^*/-\mu^*}(t^{m+1-k})P_{-\lambda^*}(1,\ldots,t^{k-1})}{P_{-\mu^*}(1,\ldots,t^{k-1}) \Pi(t^{m+1-k}; 1,\ldots,t^{k-1})} & \mu,\lambda \in \bSig_\infty^{(k)} \text{ for some }0 \leq k \leq \min(m,n) \\
    P_{\mu^*/\lambda^*}(t^m) \frac{P_{\lambda^*}(1,\ldots,t^{m-1})}{P_{\mu^*}(1,\ldots,t^{m})} & \mu \in \bSig_\infty^{(m+1)},\lambda \in \bSig_\infty^{(m)} \\
    0 & \text{otherwise}
    \end{cases}
    \end{equation} 
for any $\lambda \in \bSig_\infty$.
\end{prop}
\begin{proof}
In the case where $\len(\mu^*) = \min(m+1,n)$ so that $A$ is full-rank, the result follows by applying \cite[Theorem 1.3, Part 2]{van2020limits} (taking care that the singular numbers in that paper are the negatives of the singular numbers here). 
The non full-rank case $\len(\mu^*) < \min(m+1,n)$ follows from the full-rank case with $m+1 > n$, as in this case the rank of $A$ does not change after removing the $(m+1)\tth$ column.
\end{proof}

Because $\ESN(A) = \ESN(A^T)$, \Cref{thm:p-adic_corners} obviously holds for removing rows rather than columns after appropriately relabeling the indices. By relating matrix corners to Hall-Littlewood polynomials, \Cref{thm:p-adic_corners} provides the key to applying the results on Hall-Littlewood branching graphs to study $p$-adic random matrices. In the second case of the transition probabilities in \eqref{eq:p-adic_corners}, one immediately recognizes the cotransition probabilities of \Cref{sec:branching_graphs}. However, one now has two added features not present in that section: (1) the signatures may have infinite parts, and (2) with matrices one may remove either rows or columns, so there are in fact two (commuting) corner maps. In the next subsection, we augment the branching graph formalism and results of \Cref{sec:branching_graphs} to handle this more complicated setup. However, let us first introduce the setup of infinite matrices.

\begin{defi}
$\GL_\infty(\Z_p)$ is the direct limit $\varinjlim \GL_N(\Z_p)$ with respect to inclusions 
\begin{align*}
    \GL_N(\Z_p) &\inj \GL_{N+1}(\Z_p) \\
    A & \mapsto \begin{pmatrix} A & 0 \\ 0 & 1\end{pmatrix}
\end{align*}
Equivalently, $\GL_\infty(\Z_p) = \bigcup_{N \geq 1} \GL_N(\Z_p)$ where we identify $\GL_N(\Z_p)$ with the group of infinite matrices for which the top left $N \times N$ corner is an element of $\GL_N(\Z_p)$ and all other entries are $1$ on the diagonal and $0$ off the diagonal.
\end{defi}

The definition 
\[
\Mat_{n \times m}(\Q_p) := \left\{Z = (Z_{ij})_{\substack{1 \leq i \leq n \\ 1 \leq j \leq m}}: Z_{ij} \in \Q_p\right\}.
\]
still makes sense when $n$ or $m$ is equal to $\infty$ by replacing $1 \leq i \leq n$ with $i \in \Z_{\geq 1}$ and similarly for $m$. When $n$ or $m$ is $\infty$, $\GL_\infty(\Z_p)$ clearly acts on this space on the appropriate side.

\subsection{Auxiliary boundary results and proof of {\Cref{thm:recover_BQ}}.} \label{subsec:boundary_aux}

In this subsection we prove a similar result to \Cref{thm:boundary}, \Cref{thm:finite_boundary}, and deduce an extension to a `two-dimensional' version of the branching graph $\G$ in \Cref{thm:double_boundary}.

\begin{defi}
For each $k \geq 1$, we define a graded graph 
\[
\G^{(k)} = \bigsqcup_{n \geq 1} \G^{(k)}(n)
\]
with vertex set at each level given by $\G^{(k)}(n) = \Sig_k$. Edges are only between adjacent levels, and to each edge from $\nu \in \G^{(k)}(n+1)$ to $\lambda \in \G^{(k)}(n)$ is associated a cotransition probability
\[
\tL^{n+1}_n(\nu,\lambda) = Q_{-\lambda/-\nu}(t^n) \frac{P_{-\lambda}(1,t,\ldots,t^{k-1})}{P_{-\nu}(1,t,\ldots,t^{k-1}) \Pi(1,t,\ldots,t^{k-1}; t^n)}.
\]
We define $\tL^m_n = \tL^{n+1}_n \cdots \tL^m_{m-1}$ for general $1 \leq n < m < \infty$ as before.
\end{defi}

The next result is a version of \Cref{thm:boundary} for this smaller branching graph $\G^{(k)}$. Recall the definition of boundary from earlier in this section.

\begin{thm}\label{thm:finite_boundary}
For any $t \in (0,1)$, the boundary $\partial \G^{(k)}$ is naturally in bijection with $\Sig_k$. Under this bijection, $\mu \in \Sig_k$ corresponds to the coherent system $(M_n^\mu)_{n \geq 1}$ defined explicitly by
\begin{equation}\label{eq:finite_coherent_explicit}
    M_n^\mu(\lambda) = Q_{-\lambda/-\mu}(t^n,t^{n+1},\ldots) \frac{P_{-\lambda}(1,t,\ldots,t^{k-1})}{P_{-\mu}(1,t,\ldots,t^{k-1}) \Pi(1,t,\ldots,t^{k-1}; t^n,t^{n+1},\ldots)}
\end{equation}
for $\lambda \in \Sig_k$.
\end{thm}

Note we have simultaneously suppressed the $k$-dependence in our notation for the measure $M_n^{\mu}$ on $\Sig_k$ and abused notation by using the same for measures on $\G$ and $\G^{(k)}$, but there is no ambiguity if one knows the length of $\mu$. The proof of \Cref{thm:finite_boundary} is an easier version of the proof of \Cref{thm:boundary}, so we simply give a sketch and outline the differences.

\begin{proof}

We first prove that every extreme coherent system is of the form \eqref{eq:finite_coherent_explicit} for some $\mu \in \Sig_k$. The analogue of \Cref{thm:approximation} similarly follows from the general result \cite[Theorem 6.1]{okounkov1998asymptotics}, so there exists a regular sequence $(\mu(n))_{n \geq 1}$ approximating any extreme coherent system. Using the explicit formula \eqref{eq:finite_skew_Q} of \Cref{thm:finite_skew_fn_computation}, a naive bound as in the proof of \Cref{thm:all_boundary1} establishes that $\mu(n)_1$ is bounded above. 

The analogue of \Cref{thm:measure_vanishing}, namely that $\tL^m_n(\mu,\lambda) = 0$ and $M_n^\mu(\lambda) = 0$ if there exists an $x$ for which $\lambda_x' > \mu_x'$, holds similarly by the branching rule. Using this one obtains that a regular sequence $(\mu(n))_{n \geq 1}$ must have last parts $\mu(n)_k$ bounded below. Together with the upper bound this yields that $(\mu(n))_{n \geq 1}$ has a convergent subsequence, where here convergence simply means that all terms in the subsequence are equal to the same $\mu \in \Sig_k$. It now follows as in the proof of \Cref{thm:all_boundary1} that in fact the coherent system approximated by $(\mu(n))_{n \geq 1}$ must be $(M^\mu_n)_{n \geq 1}$ for this $\mu$. 

It remains to prove that every coherent system of the form \eqref{eq:finite_coherent_explicit} is in fact extreme. The proof is the same as that of \Cref{thm:all_boundary2} using the above analogue of \Cref{thm:measure_vanishing}, except that no measure-theoretic details are necessary because the decomposition of an arbitrary coherent system into extreme ones takes the form of a sum over the countable set $\Sig_k$. 
\end{proof}

For applications in the next section it is desirable to in some sense combine $\G$ and $\G^{(k)}$ by working with extended signatures. We wish to define a doubly-graded graph with cotransition probabilities which generalize the earlier $L^{n+1}_n, \tL^{n+1}_n$ and which correspond to the situation of removing rows and columns from a matrix in \Cref{thm:p-adic_corners}.

\begin{defi}
Define
\[
\GG = \bigsqcup_{m,n \geq 1} \GG(m,n)
\]
with $\GG(m,n) = \bSig_\infty$ for each $m,n$, and edges from $\GG(m+1,n)$ to $\GG(m,n)$ with weights
\begin{equation}\label{eq:twoD_L}
     L^{m+1,n}_{m,n}(\mu,\lambda) = \begin{cases}
     \frac{Q_{-\lambda^*/-\mu^*}(t^{m+1-k})P_{-\lambda^*}(1,\ldots,t^{k-1})}{P_{-\mu^*}(1,\ldots,t^{k-1}) \Pi(t^{m+1-k}; 1,\ldots,t^{k-1})} & \mu,\lambda \in \bSig_\infty^{(k)} \text{ for some }0 \leq k \leq \min(m,n) \\
    P_{\mu^*/\lambda^*}(t^m) \frac{P_{\lambda^*}(1,\ldots,t^{m-1})}{P_{\mu^*}(1,\ldots,t^{m})} & \mu \in \bSig_\infty^{(m+1)},\lambda \in \bSig_\infty^{(m)} \\
    0 & \text{otherwise}
    \end{cases}
\end{equation}
and edges from $\GG(m,n+1)$ to $\GG(m,n)$ with weights $L^{m,n+1}_{m,n}(\mu,\lambda) = L^{n+1,m}_{n,m}(\mu,\lambda)$. 
\end{defi}

It follows immediately from the Cauchy identity \Cref{thm:finite_cauchy} that \[
L^{m+1,n}_{m,n}L^{m+1,n+1}_{m+1,n} = L^{m,n+1}_{m,n} L^{m+1,n+1}_{m,n+1},
\] so there is no ambiguity in defining coherent systems of probability measures on $\GG$.

\begin{thm}\label{thm:double_boundary}
For $t \in (0,1)$, the boundary $\partial \GG$ is in bijection with $\bSig_\infty$. The extreme coherent system $(M^\mu_{m,n})_{m,n \geq 1}$ corresponding to $\mu \in \Sig_\infty$ is determined by
\[
M^\mu_{m,n}(\nu) = \sum_{\lambda \in \Sig_n} M^{\mu}_n(\lambda) M^{\lambda}_{m-n+1}(\nu^*)
\]
for $m \geq n$ and hence for all $m,n$ by coherency. The extreme coherent system corresponding to $\mu \in \Sig_\infty^{(k)}$ is determined by 
\[
M^\mu_{m,n}(\nu) = \sum_{\lambda \in \Sig_k} M^{\mu^*}_{n-k+1}(\lambda) M^\lambda_{m-k+1}(\nu^*)
\]
for $m,n \geq k$ and hence for all $m,n$ by coherency.
\end{thm}
\begin{proof}
First note that every coherent system on $\GG$ is determined by a sequence of coherent systems on the subgraphs with vertex sets
\begin{equation}\label{eq:row_subgraph}
    \bigsqcup_{m \geq n} \GG(m,n)
\end{equation}
for $n \geq 1$, which are themselves coherent with one another under the links $L^{m,n+1}_{m,n}$. By the definition of the cotransition probabilities \eqref{eq:twoD_L}, a coherent system on \eqref{eq:row_subgraph} must decompose as a convex combination of $n+1$ coherent systems, each one having all measures supported on $\bSig_\infty^{(k)}$ for $0 \leq k \leq n$. Hence extreme coherent systems on \eqref{eq:row_subgraph} are parametrized by $\bSig_k$ by applying \Cref{thm:finite_boundary} for each $k$. 

It follows by the above-mentioned commutativity $L^{m+1,n}_{m,n}L^{m+1,n+1}_{m+1,n} = L^{m,n+1}_{m,n} L^{m+1,n+1}_{m,n+1}$ that given a coherent system $(M_m)_{m \geq n}$ on the graph \eqref{eq:row_subgraph}, $(M_m L^{m,n}_{m,n-1})_{m \geq n}$ is a coherent system on 
\[
\bigsqcup_{m \geq n} \GG(m,n-1).
\]
Since $L^{m,n}_{m,n-1}$ takes coherent systems to coherent systems, by decomposing these into extreme coherent systems it induces a map  $\cM(\bSig_n) \to \cM(\bSig_{n-1})$ between spaces of probability measures on the respective boundaries, i.e. a Markov kernel. It follows from the explicit formulas \eqref{eq:finite_coherent_explicit}, \eqref{eq:twoD_L} and the Cauchy identity \Cref{thm:finite_cauchy} that this Markov map is itself given by $L^{m,n}_{m,n-1}$ on the appropriately restricted domain, after identifying $\bSig_n$ and $\bSig_{n-1}$ as subsets of $\bSig_\infty$ in the obvious way. 

Hence $\partial \GG$ is in bijection with coherent systems on the graph with vertex set
\[
\bigsqcup_{n \geq 1} \bSig_n
\]
and edges between $n\tth$ and $(n-1)^{\text{st}}$ level given by $L^{m,n}_{m,n-1}$ for any $m \geq n$ (note the these links are independent of $m \geq n$ by \eqref{eq:twoD_L}). The boundary of this graph is classified by $\bSig_\infty$ by combining \Cref{thm:boundary} (for coherent systems supported on $\Sig_\infty$) and \Cref{thm:finite_boundary} (for coherent systems supported on $\bSig_\infty^{(k)}$), and the explicit coherent systems in the statement follow from the above computations.
\end{proof}

\begin{proof}[Proof of {\Cref{thm:recover_BQ}}]
Any $\GL_\infty(\Z_p) \times \GL_\infty(\Z_p)$-invariant measure on $\Mat_{\infty \times \infty}(\Q_p)$ is uniquely determined by its marginals on $m \times n$ truncations for finite $m,n$, which are each $\GL_m(\Z_p) \times \GL_n(\Z_p)$-invariant. The $\GL_n(\Z_p) \times \GL_m(\Z_p)$-invariant probability measures on $\Mat_{n \times m}(\Q_p)$ are in bijection with probability measures on $\bSig_\infty$ supported on signatures with at most $\min(m,n)$ finite parts. Hence removing a row (resp. column) induces a Markov kernel $\cM(\bSig_\infty) \to \cM(\bSig_\infty)$, and by \Cref{thm:p-adic_corners} this Markov kernel is exactly $L^{m,n}_{m,n-1}$ (resp. $L^{m,n}_{m-1,n}$). Hence \Cref{thm:double_boundary} yields that the set of extreme $\GL_\infty(\Z_p) \times \GL_\infty(\Z_p)$-invariant measures on $\Mat_{\infty \times \infty}(\Q_p)$ is in bijection with $\bSig_\infty$. Here the measure $\EM_\mu$ corresponding to $\mu$ is determined by the fact that each $m \times n$ corner has singular numbers distributed by the measure $M_{m,n}^\mu$ in \Cref{thm:double_boundary}.
\end{proof}

We have shown that the extreme bi-invariant measures are parametrized somehow by $\bSig_\infty$, but in \cite{bufetov2017ergodic} the measure corresponding to a given $\mu \in \bSig_\infty$ is defined quite differently, and it is not at all clear a priori that it is the same as our measure $\EM_\mu$. Let us describe these measures.

In the finite or infinite setting, there are two natural families of random matrices in $\Mat_{n \times m}(\Q_p)$ which are invariant under the natural action of $\GL_n(\Z_p) \times \GL_m(\Z_p)$:
\begin{itemize}
    \item (Haar) $p^{-k} Z$, where $k \in \Z \cup \{-\infty\}$ and $Z$ has iid entries distributed by the additive Haar measure on $\Z_p$.
    \item (Nonsymmetric Wishart-type) $p^{-k}X^T Y$, where $X \in \Z_p^n, Y \in \Z_p^m$ have iid additive Haar entries. 
\end{itemize}
One can of course obtain invariant measures by summing the above random matrices, which motivates the following class of measures. 

\begin{defi}\label{def:emu}
Let $\mu \in \bSig_\infty$, and let $\mu_\infty := \lim_{\ell \to \infty} \mu_\ell \in \Z \cup \{-\infty\}$. Let $X^{(\ell)}_i, Y^{(\ell)}_j, Z_{ij}$ be iid and distributed by the additive Haar measure on $\Z_p$ for $i,j,\ell \geq 1$. Then we define the measure $\tE_\mu$ on $\Mat_{\infty \times \infty}(\Q_p)$ as the distribution of the random matrix
\[
\left(\sum_{\ell: \mu_\ell > \mu_\infty} p^{-\mu_\ell} X^{(\ell)}_i Y^{(\ell)}_j + p^{-\mu_\infty} Z_{ij}\right)_{i,j \geq 1}.
\]
\end{defi}

It is shown in \cite[Theorem 1.3]{bufetov2017ergodic} that the $\tE_\mu, \mu \in \bSig_\infty$ are exactly the extreme $\GL_\infty(\Z_p) \times \GL_\infty(\Z_p)$-invariant measures on $\Mat_{\infty \times \infty}(\Q_p)$.

\begin{prop}\label{thm:emu=temu}
For any $\mu \in \bSig_\infty$, $\tE_\mu = \EM_\mu$.
\end{prop}
\begin{proof}
By combining \Cref{thm:recover_BQ} with the result \cite[Theorem 1.3]{bufetov2017ergodic} that the $\tE_\mu$ are exactly the extreme measures, we have that $\{\tE_\mu: \mu \in \bSig_\infty\} = \{\EM_\mu: \mu \in \bSig_\infty\}$. Hence for each $\mu \in \bSig_\infty$ there exists $\nu \in \bSig_\infty$ such that $\tE_\mu = \EM_\nu$. Suppose for the sake of contradiction that $\nu \neq \mu$. Let $k \geq 1$ be the smallest index for which $\mu_k \neq \nu_k$, let 
\[
f: \Mat_{\infty \times \infty}(\Q_p) \to \bSig_k
\]
be the map to the first $k$ singular numbers of the top left $k \times k$ corner, and let 
\[
S^{(k)}_{\leq \mu} := \{\lambda \in \bSig_k: \lambda_i \leq \mu_i \text{ for }1 \leq i \leq k\}.
\]
We claim that 
\begin{equation}
    \label{eq:E_forward}
    f_*(\tE_\mu) \text{ is supported on }S^{(k)}_{\leq \mu} \text{ and }(f_*(\tE_\mu))(\mu_1,\ldots,\mu_k) > 0,
\end{equation}
and 
\begin{equation}
    \label{eq:tE_forward}
    f_*(\EM_\nu) \text{ is supported on }S^{(k)}_{\leq \nu}\text{ and }(f_*(\EM_\nu))(\nu_1,\ldots,\nu_k) > 0. 
\end{equation}

The first, \eqref{eq:E_forward}, follows straightforwardly from \Cref{def:emu}, while \eqref{eq:tE_forward} follows from \Cref{thm:double_boundary} and \Cref{thm:measure_vanishing}. 

If $\mu_k > \nu_k$, then $\Supp(f_*(\tE_\mu)) \supsetneq \Supp(f_*(\EM_\nu))$, while if $\mu_k < \nu_k$ then $\Supp(f_*(\tE_\mu)) \subsetneq \Supp(f_*(\EM_\nu))$, contradicting the claim $f_*(\tE_\mu) = f_*(\EM_\nu)$. Therefore there does not exist $k$ as above, so $\mu = \nu$, completing the proof.
\end{proof}

Combining \Cref{thm:emu=temu} with \Cref{thm:recover_BQ} in fact provides a (quite indirect!) computation of the singular numbers of $m \times n$ truncations of the infinite matrices in \Cref{def:emu}. 

\begin{cor}\label{thm:computed_sns}
The singular numbers of an $m \times n$ corner of an infinite matrix with distribution $\tE_\mu$ are distributed by the measure $M_{m,n}^\mu$ of \Cref{thm:double_boundary}. 
\end{cor}

It seems possible that the summation which defines the measures $M_{m,n}^\mu$ may be simplified to get more explicit formulas for the above distributions, though we do not address this question here. 

\begin{rmk}\label{rmk:BQ_differences}
There are several comments on the relation between our setup and that of \cite{bufetov2017ergodic} which are worth highlighting:
\begin{itemize}
    \item We work over $\Q_p$ while \cite{bufetov2017ergodic} works over an arbitrary non-Archimedean local field $F$. Such a field has a ring of integers $\mathcal{O}_F$ playing the role of $\Z_p$ and a uniformizer $\omega$ playing the role of $p$, and a finite residue field $\mathcal{O}_F/\omega \mathcal{O}_F \cong \F_q$. Our results transfer mutatis mutandis to this setting with $t=1/q$, as the only needed input \Cref{thm:p-adic_corners} transfers in view of \cite[Remark 4]{van2020limits}. 
    \item While we simply prove a bijection, a short additonal argument shows that the space of extreme invariant measures on $\Mat_{\infty \times \infty}(\Q_p)$ is homeomorphic to $\bSig_\infty$ with natural topologies on both spaces, see the proof of Theorem 1.3 of \cite{bufetov2017ergodic} for details. 
    \item We have used the language of extreme and ergodic measures interchangeably, but for an explanation of how the extreme measures are exactly the ergodic ones in the conventional sense, for this problem and more general versions, see \cite[Section 2.1]{bufetov2017ergodic}.
\end{itemize}
\end{rmk}

\begin{rmk}\label{rmk:symm_HL_proof}
As mentioned in the Introduction, \cite{bufetov2017ergodic} also classify extreme measures on infinite symmetric matrices $\Sym(\N,\Q_p):=\{A \in \Mat_{\infty \times \infty}(\Q_p): A^T = A\}$ invariant under the action of $\GL_\infty(\Z_p)$ by $(g,A) \mapsto gAg^T$. The statement is more involved, essentially due to the fact that the $\GL_n(\Z_p)$-orbits on $\Sym(n,\Q_p)$ are parametrized by their singular numbers together with additional data, unlike the $\GL_n(\Z_p) \times \GL_m(\Z_p)$-orbits on $\Mat_{n \times m}(\Z_p)$. To pursue a similar strategy to our proof of \Cref{thm:recover_BQ} one would need an analogue of \Cref{thm:p-adic_corners}, i.e. a result giving the distribution of the $\GL_{n-1}(\Z_p)$-orbit of an $(n-1) \times (n-1)$ corner of an $n \times n$ symmetric matrix drawn uniformly from a fixed $\GL_n(\Z_p)$-orbit. Given that the parametrization of these orbits involves more data than the (extended) signature specifying their singular numbers, it is not immediately clear how the answer would be expressed in terms of Hall-Littlewood polynomials.

We do however expect a solution in terms of Hall-Littlewood polynomials to a related problem which is coarser. The problem is to find the distribution of just the singular numbers, rather than $\GL_{n-1}(\Z_p)$-orbits, of an $(n-1) \times (n-1)$ corner of a random element of $\Sym(n,\Q_p)$ with fixed singular numbers and $\GL_n(\Z_p)$-invariant distribution. The existence of such a result is suggested by a known expression for the singular numbers of an $n \times n$ symmetric matrix with iid (apart from the symmetry constraint) entries distributed by the additive Haar measure on $\Z_p$. This distribution was computed in \cite{clancy2015cohen}, and shown to be equivalent to a measure coming from one of the so-called Littlewood identities for Hall-Littlewood polynomials in \cite{fulman2016hall}. It seems natural that a solution to this problem could be augmented with the extra data required to parametrize $\GL_n(\Z_p)$-orbits, answering the question of the previous paragraph. We have not attempted to pursue this direction.
\end{rmk}

\section{Ergodic decomposition of {$p$}-adic Hua measures} \label{sec:phua}

We now define a special family of measures on $\Mat_{\infty \times \infty}(\Q_p)$, the \emph{$p$-adic Hua measures}, introduced in \cite{neretin2013hua}. Their decomposition into the ergodic measures $\tE_\mu$ of \Cref{def:emu} was computed in \cite{assiotis2020infinite}. We will rederive that result, showing in the process that the $p$-adic Hua measures have a natural interpretation in terms of measures on partitions derived from Hall-Littlewood polynomials.

\begin{defi}\label{def:pos_part_sig}
For $\lambda \in \Sig_n$, we set 
\[
\lambda^+ := (\max(\lambda_1,0),\ldots,\max(\lambda_n,0)) \in \Sig_n^{\geq 0}.
\]
\end{defi}

\begin{defi}\label{def:p-hua}
The $p$-adic Hua measure $\bbM_n^{(s)}$ on $\Mat_{n \times n}(\Q_p)$ is defined by 
\[
d\bbM_n^{(s)}(A) = \frac{(p^{-1-s};p^{-1})_n^2}{(p^{-1-s};p^{-1})_{2n}} p^{|\ESN(A)^+|(-s-2n)} d\M^{(n)}(A),
\]
where $\M^{(n)}$ is the product over all $n^2$ matrix entries of the additive Haar measure $\M$ on $\Q_p$.
\end{defi}

The following computation of the distribution of the singular numbers of $\bbM_n^{(s)}$ is done in \cite[Proposition 3.1]{assiotis2020infinite}, using \Cref{def:p-hua} and results of \cite[Chapter V]{mac}.

\begin{prop}\label{thm:finite_phua_SNs}
The pushforward of $\bbM_n^{(s)}$ under $\SN: \Mat_{n \times n}(\Q_p) \to \bSig_n$ is supported on $\Sig_n$ and given by
\[
\left(\SN_*(\bbM_n^{(s)})\right)(\lambda) = \frac{(u;t)_n^2}{(u;t)_{2n}} u^{|\lambda^+|} t^{(2n-1)(|\lambda^+| - |\lambda|) + 2n(\lambda)} \frac{(t;t)_n^2}{\prod_{x \in \Z} (t;t)_{m_x(\lambda)}},
\]
where as usual $t=1/p$, and $u=t^{1+s}$.
\end{prop}

We may now prove the main result, which we recall. Note that by \Cref{thm:emu=temu} the same result holds with $E_\mu$ replaced by $\tE_\mu$, and it is the latter version which was proven in \cite{assiotis2020infinite}.

\recoverA*
\begin{proof}
The $p$-adic Hua measure is uniquely determined by its projections to $n \times n$ corners, and by extremality of the measures $E_\mu$ any decomposition into a convex combination of them is unique. Hence it suffices to show that a matrix $A$, distributed by the measure on $\Mat_{\infty \times \infty}(\Q_p)$ described by RHS\eqref{eq:phua_decomp}, has $n \times n$ corners given by the finite $p$-adic Hua measure $\bbM_n^{(s)}$. By \Cref{thm:finite_phua_SNs} and \Cref{thm:recover_BQ}, it suffices to show
\begin{align} \label{eq:assiotis_hl_suffices}
\begin{split}
    &\sum_{\mu \in \Y} \frac{P_\mu(1,t,\ldots) Q_\mu(u,ut,\ldots)}{\Pi(1,\ldots;u,\ldots)} \sum_{\lambda \in \Sig_n^{\geq 0}} \frac{P_{\mu/\lambda}(t^n,\ldots)P_\lambda(1,\ldots,t^{n-1})}{P_\mu(1,t,\ldots)} Q_{-\nu/-\lambda}(t,t^2,\ldots) \\
    \times &\frac{P_{-\nu}(1,\ldots,t^{n-1})}{P_{-\lambda}(1,\ldots,t^{n-1}) \Pi(1,\ldots,t^{n-1};t,\ldots)} = \frac{(u;t)_n^2}{(u;t)_{2n}} u^{|\nu^+|} t^{(2n-1)(|\nu^+| - |\nu|) + 2n(\nu)} \frac{(t;t)_n^2}{\prod_{x \in \Z} (t;t)_{m_x(\nu)}}
\end{split}
\end{align}
The proof is a surprisingly long series of applications of the Cauchy identity/branching rule and principal specialization formulas. We first cancel the $P_\mu(1,\ldots)$ factors and apply the Cauchy identity \eqref{eq:infinite_cauchy} to the sum over $\mu$ to obtain
\begin{multline}\label{eq:hl_intermediate_1}
    \frac{P_{-\nu}(1,\ldots,t^{n-1})}{\Pi(1,\ldots;u,\ldots) \Pi(1,\ldots,t^{n-1}; t,\ldots)} \\
    \times \sum_{\lambda \in \Sig_n^{\geq 0}} \frac{P_\lambda(1,\ldots,t^{n-1})}{P_{-\lambda}(1,\ldots,t^{n-1})} Q_{-\nu/-\lambda}(t,\ldots) Q_{\lambda/(0[n])}(u,\ldots) \Pi(t^n,\ldots;u,\ldots).
\end{multline}
Using that 
\[
\Pi(1,\ldots,t^{n-1};t,\ldots) = (t;t)_n,
\]
and 
\[
P_{-\nu}(1,\ldots,t^{n-1}) = P_\nu(1,\ldots,t^{-(n-1)}) = t^{-(n-1)|\nu|} P_\nu(1,\ldots,t^{n-1})
\]
and similarly for $\lambda$, \eqref{eq:hl_intermediate_1} becomes 
\begin{equation}\label{eq:hl_intermediate_2}
    \frac{(t;t)_n P_\nu(1,\ldots,t^{n-1}) t^{(n-1)(|\lambda|-|\nu|)}}{\Pi(1,\ldots,t^{n-1}; u,\ldots)}\sum_{\lambda \in \Sig_n^{\geq 0}} Q_{-\nu/-\lambda}(t,\ldots)Q_{\lambda/(0[n])}(u,\ldots).
\end{equation}
It follows from the explicit branching rule \Cref{def:skew_fns} and the principal specialization formula \Cref{thm:hl_principal_formulas} for $P$ that
\begin{equation}\label{eq:Q_negate_sigs}
    Q_{-\nu/-\lambda}(x) = Q_{\lambda/\nu}(x) \frac{t^{-n(\lambda)}P_\lambda(1,\ldots,t^{n-1})}{t^{-n(\nu)}P_\nu(1,\ldots,t^{n-1})}.
\end{equation}
By definition of skew $Q$ functions \eqref{eq:Q_negate_sigs} immediately extends to 
\begin{equation*}
    Q_{-\nu/-\lambda}(x_1,\ldots,x_k) = Q_{\lambda/\nu}(x_1,\ldots,x_k) \frac{t^{-n(\lambda)}P_\lambda(1,\ldots,t^{n-1})}{t^{-n(\nu)}P_\nu(1,\ldots,t^{n-1})}
\end{equation*}
for any $k$, hence to an equality of symmetric functions and hence specializes to 
\begin{equation}\label{eq:Q_negate_sigs_spec}
    Q_{-\nu/-\lambda}(t^n,\ldots) = Q_{\lambda/\nu}(t^n,\ldots) \frac{t^{-n(\lambda)}P_\lambda(1,\ldots,t^{n-1})}{t^{-n(\nu)}P_\nu(1,\ldots,t^{n-1})}.
\end{equation}
By first absorbing the $t^{(n-1)(|\lambda|-|\nu|)}$ into $Q_{-\nu/-\lambda}$ in \eqref{eq:hl_intermediate_2} and then substituting \eqref{eq:Q_negate_sigs_spec} and simplifying $Q_{\lambda/(0[n])}$ via \Cref{thm:hl_principal_formulas}, \eqref{eq:hl_intermediate_2} becomes
\begin{align}\label{eq:hl_intermediate_3}
\begin{split}
    &\frac{(t;t)_n P_\nu(1,\ldots,t^{n-1})}{\Pi(1,\ldots,t^{n-1};u,\ldots)} \sum_{\lambda \in \Sig_n^{\geq 0}} Q_{\lambda/\nu}(t^n,\ldots) \frac{t^{-n(\lambda)}P_\lambda(1,\ldots,t^{n-1})}{t^{-n(\nu)}P_\nu(1,\ldots,t^{n-1})} u^{|\lambda|} t^{n(\lambda)} \\
    &=\frac{(t;t)_n t^{n(\nu)}}{\Pi(1,\ldots,t^{n-1};u,\ldots)} \sum_{\lambda \in \Sig_n^{\geq 0}} Q_{\lambda/\nu}(t^n,\ldots)P_\lambda(u,\ldots,ut^{n-1}).
\end{split}
\end{align}
At first glance, the sum on the RHS of \eqref{eq:hl_intermediate_3} looks like the one in the Cauchy identity \eqref{eq:finite_cauchy}, but there is a nontrivial difference: the sum is over only nonnegative signatures. If $\nu \in \Sig_n^{\geq 0}$ itself, this poses no issue and the Cauchy identity applies directly, but in general this is not the case. 

Luckily, using the explicit formula in \Cref{thm:skew_formula} we may relate the sum in \eqref{eq:hl_intermediate_3} to one to which the Cauchy identity applies. By slightly rearranging terms in \Cref{thm:skew_formula}, we have that for $\lambda \in \Sig_n^{\geq 0}$,
\begin{align}\label{eq:skewQ_special_case}
\begin{split}
    Q_{\lambda/\nu}(t^n,\ldots) &= \frac{t^{n \cdot (|\lambda|-|\nu|)}}{\prod_{x \in \Z} (t;t)_{m_x(\nu)}} \prod_{x \leq 0} (t^{1+n-\nu_x'};t)_{m_x(\nu)} t^{\binom{n-\nu_x'}{2}} \prod_{x > 0} (t^{1+\lambda_x'-\nu_x'};t)_{m_x(\nu)} t^{\binom{\lambda_x'-\nu_x'}{2}} \\
    Q_{\lambda/\nu^+}(t^n,\ldots) &= \frac{t^{n \cdot (|\lambda| - |\nu^+|)}}{\prod_{x > 0}(t;t)_{m_x(\nu)}} \prod_{x > 0} (t^{1+\lambda_x' - \nu_x'};t)_{m_x(\nu)} t^{\binom{\lambda_x'-\nu_x'}{2}}
\end{split}
\end{align}
where $\nu^+$ is the truncation as in \Cref{def:pos_part_sig}. Since
\[
\prod_{x \leq 0} (t^{1+n-\nu_x'};t)_{m_x(\nu)} = (t;t)_{|\{i: \nu_i \leq 0\}|} = (t;t)_{m_0(\nu^+)},
\]
\eqref{eq:skewQ_special_case} implies that
\[
Q_{\lambda/\nu}(t^n,\ldots) = t^{n \cdot(|\nu^+| - |\nu|) + \sum_{x \leq 0} \binom{n-\nu_x'}{2}} \frac{(t;t)_{m_0(\nu^+)}}{\prod_{x \leq 0} (t;t)_{m_x(\nu)}} Q_{\lambda/\nu^+}(t^n,\ldots).
\]
Therefore 
\begin{align}\label{eq:truncate_sig_sum}
\begin{split}
    &\sum_{\lambda \in \Sig_n^{\geq 0}} Q_{\lambda/\nu}(t^n,\ldots)P_\lambda(u,\ldots,ut^{n-1}) \\
    &=t^{n \cdot(|\nu^+| - |\nu|) + \sum_{x \leq 0} \binom{n-\nu_x'}{2}} \frac{(t;t)_{m_0(\nu^+)}}{\prod_{x \leq 0} (t;t)_{m_x(\nu)}} \sum_{\lambda \in \Sig_n^{\geq 0}} Q_{\lambda/\nu^+}(t^n,\ldots) P_\lambda(u,\ldots,ut^{n-1}) \\
    &= t^{n \cdot(|\nu^+| - |\nu|) + \sum_{x \leq 0} \binom{n-\nu_x'}{2}} \frac{(t;t)_{m_0(\nu^+)}}{\prod_{x \leq 0} (t;t)_{m_x(\nu)}} \Pi(t^n,\ldots;u,\ldots,ut^{n-1}) P_{\nu^+}(u,\ldots,ut^{n-1}) \\
    &= \frac{(t;t)_n}{(ut^n;t)_n} \frac{u^{|\nu^+|} t^{n \cdot(|\nu^+| - |\nu|) + \sum_{x \leq 0} \binom{n-\nu_x'}{2} + n(\nu^+)}}{\prod_{x \in \Z} (t;t)_{m_x(\nu)}}
\end{split}
\end{align}
by applying \eqref{eq:finite_cauchy} and \Cref{thm:hl_principal_formulas}. It is an elementary check from the definitions that 
\begin{equation}\label{eq:sig_identity}
    n \cdot(|\nu^+| - |\nu|) + \sum_{x \leq 0} \binom{n-\nu_x'}{2} + n(\nu^+) = (2n-1)(|\nu^+| - |\nu|) + n(\nu).
\end{equation}
Substituting \eqref{eq:sig_identity} into \eqref{eq:truncate_sig_sum} and the result into \eqref{eq:hl_intermediate_3} yields 
\begin{align}
\begin{split}
    &(t;t)_n (u;t)_n t^{n(\nu)}\frac{(t;t)_n}{(ut^n;t)_n} \frac{u^{|\nu^+|} t^{(2n-1)(|\nu^+| - |\nu|) + n(\nu)}}{\prod_{x \in \Z} (t;t)_{m_x(\nu)}} \\
    &= \frac{(u;t)_n^2}{(u;t)_{2n}} u^{|\nu^+|} t^{(2n-1)(|\nu^+| - |\nu|) + 2n(\nu)} \frac{(t;t)_n^2}{\prod_{x \in \Z} (t;t)_{m_x(\nu)}},
\end{split}
\end{align}
which is the formula in \Cref{thm:finite_phua_SNs}, completing the proof.
\end{proof}

In some sense, the interpretation of the measures $M_n^{(s)}$ which we have given here explains their special nature and gives a natural non-historical route to their discovery. Let us suppose that one knew only \Cref{thm:p-adic_corners} and \Cref{thm:recover_BQ}, and wished to look for family of measures on $\Mat_{n \times n}(\Q_p)$ which are consistent under taking corners. Any measure on the boundary yields such a family (and vice versa), but only for very nice measures on the boundary do we expect the resulting measure on corners to have any reasonable description. Because the cotransition probabilities feature principal specializations, the natural candidate for this measure on the boundary is a Hall-Littlewood measure with two principal specializations $u_1,u_1t,\ldots$ and $u_2,u_2t,\ldots$. Indeed, the above combinatorics would break down entirely for other Hall-Littlewood measures. This leaves one free parameter because one may divide one specialization and multiply the other by any positive real number without changing the measure, and this free parameter is exactly the one in the $p$-adic Hua measure. 

In another direction we note that, if one did not already know the result of \cite{assiotis2020infinite}, the above considerations could help guess it. Since known natural measures on finite $p$-adic matrices have singular numbers distributed by Hall-Littlewood measures by \cite[Theorem 1.3 and Corollary 1.4]{van2020limits}, and the ergodic decomposition of a measure on infinite matrices is the analogue of the distribution of singular numbers of a finite matrix, it is natural to search for the ergodic decomposition within the space of Hall-Littlewood measures. As mentioned above, essentially the only Hall-Littlewood measures with nice explicit densities are those with principal specializations, of finite or infinite length. If one were of finite length, say $N$, then it is a straightforward consequence of \Cref{thm:recover_BQ} that at most $N$ singular numbers of any corner are nonzero, which contradicts \Cref{thm:finite_phua_SNs}. Hence if the ergodic decomposition is according to a well-behaved (principally specialized) Hall-Littlewood measure, both specializations must be infinite, and this leads exactly to the one-parameter family of Hall-Littlewood measures which do indeed appear.

\section{Products of finite {$p$}-adic random matrices} \label{sec:matrix_products}

In this section, we prove the exact formula \Cref{thm:product_intro} for singular numbers of products of Haar matrices. Recall that when $n \leq m$ and $A \in \Mat_{n \times m}(\Z_p)$ is nonsingular, the image $\Im(A) \subset \Z_p^n$ of the map $A: \Z_p^m \to \Z_p^n$ is a $\Z_p$-submodule and the cokernel $\coker(A) := \Z_p^n/\Im(A)$ is a finite abelian $p$-group, given by
\[
\coker(A) \cong \bigoplus_{i=1}^n \Z/p^{\lambda_i}\Z
\]
where $-\lambda = \SN(A)$. Most literature on $p$-adic random matrices takes this perspective of random abelian $p$-groups, see the references in the Introduction. 

To relate cokernels/singular numbers of matrix products to Hall-Littlewood combinatorics, we quote a special case of \cite[Corollary 3.4]{van2020limits}, which states that (negative\footnote{Let us reiterate that the sign convention on singular numbers here is opposite from the one in \cite{van2020limits} from which the above was taken, so \Cref{thm:hl_from_p_paper} differs from the statement in \cite[Corollary 3.4]{van2020limits} by a sign.}) singular numbers of matrix products are distributed as a Hall-Littlewood process (defined earlier in \eqref{eq:general_hl_proc}). In what follows we use the notation $a[k]$ for variables repeated $k$ times. 

\begin{prop}\label{thm:hl_from_p_paper}
Let $t=1/p$, fix $n \geq 1$, and for $1 \leq i \leq k$ let $A_i$ have iid entries distributed by the additive Haar measure on $\Z_p$. Then for $\lambda^{(1)},\ldots,\lambda^{(k)} \in \Sig_n^{\geq 0}$,
\begin{multline*}
    \Pr(\SN(A_i \cdots A_1) = -\lambda^{(i)}\text{ for all }i=1,\ldots,k)  = \frac{P_{\lambda^{(k)}}(1,\ldots,t^{n-1}) \prod_{i=1}^k Q_{\lambda^{(i)}/\lambda^{(i-1)}}(t,t^2,\ldots) }{\Pi(1,\ldots,t^{n-1}; t[k],t^2[k],\ldots)}
\end{multline*}
where we take $\lambda^{(0)} = (0[n])$.
\end{prop}

We will deduce \Cref{thm:product_intro} from \Cref{thm:hl_from_p_paper} together with the following, which uses results of \Cref{sec:skew_formulas} to write an explicit formula for Hall-Littlewood process dynamics. 

\begin{prop}\label{thm:explicit_cauchy_dynamics}
For $n \geq 1$ and $\lambda,\nu \in \Sig_n$, we have 
\begin{equation}
    \frac{Q_{\nu/\lambda}(u,ut,\ldots) P_\nu(1,\ldots,t^{n-1})}{P_\lambda(1,\ldots,t^{n-1}) \Pi(1,\ldots,t^{n-1}; u,ut,\ldots)} = (u;t)_n u^{|\nu|-|\lambda|} t^{n(\nu)-n(\lambda) + n(\nu/\lambda)} \prod_{x \in \Z} \sqbinom{\nu_x'-\lambda_{x+1}'}{\nu_x' - \nu_{x+1}'}_t.
\end{equation}
\end{prop}
\begin{proof}
It follows from the definition in \eqref{eq:def_cauchy_kernel} and telescoping that
\[
\frac{1}{\Pi(1,\ldots,t^{n-1}; u,ut,\ldots)} = (u;t)_n.
\]
Combining \Cref{thm:skew_formula} with \Cref{thm:hl_principal_formulas} yields 
\[
\frac{Q_{\nu/\lambda}(u,ut,\ldots) P_\nu(1,\ldots,t^{n-1})}{P_\lambda(1,\ldots,t^{n-1})} = u^{|\nu|-|\lambda|} t^{n(\nu/\lambda)+n(\nu)-n(\lambda)}\prod_{x \in \Z} \frac{(t^{1+\nu_x'-\lambda_x'};t)_{m_x(\lambda)}}{(t;t)_{m_x(\nu)}} 
\]
Noting that 
\[
\prod_{x \in \Z} \frac{(t^{1+\nu_x'-\lambda_x'};t)_{m_x(\lambda)}}{(t;t)_{m_x(\nu)}} = \prod_{x \in \Z} \sqbinom{\nu_x'-\lambda_{x+1}'}{\nu_x' - \nu_{x+1}'}_t
\]
completes the proof.
\end{proof}

\begin{proof}[Proof of {\Cref{thm:product_intro}}]
Follows immediately by combining \Cref{thm:hl_from_p_paper} and \Cref{thm:explicit_cauchy_dynamics} with $u=t$.
\end{proof}

\appendix 

\section{Markov dynamics on the boundary}\label{sec:appendix_markov}

For finite $n$, one has natural Hall-Littlewood process dynamics on $\Sig_n$, see \eqref{eq:general_cauchy_dynamics}. It is natural to ask whether these yield dynamics on the boundary $\Sig_\infty$, and whether anything interesting may be said about them. For the $q$-Gelfand-Tsetlin graph mentioned in the Introduction, the resulting dynamics on $\Sig_\infty$ were studied in \cite{borodin2013markov}, see also the references therein for previously studied instances of this question on branching graphs in which the boundary is continuous rather than discrete. In this Appendix, we show in \Cref{thm:links_commutes} that the Hall-Littlewood process dynamics on the levels of $\G$ indeed lift to dynamics on $\partial \G$. This is motivated by a parallel work \cite{van2021q} which studies a continuous-time limit of these dynamics; we are not presently aware of an interpretation in terms of $p$-adic random matrices when $t=1/p$, as with earlier results in this paper. 

While the fact that the dynamics in \cite{van2021q} may be viewed as dynamics on $\partial \G$ is not technically necessary for their analysis in \cite{van2021q}, it provides an interesting context for the results of \cite{van2021q}. There exist other dynamics which arise in a structurally similar manner for different degenerations of Macdonald polynomials, but nonetheless have quite different asymptotic behavior, see the introduction to \cite{van2021q} for further discussion and references. Because \Cref{thm:links_commutes} requires branching graph formalism which is orthogonal to \cite{van2021q} apart from this motivation, we chose to prove it here and discuss the statement informally in \cite{van2021q}.

We now consider Markovian dynamics on the boundary $\partial \G$. We will show that the dynamics \eqref{eq:general_cauchy_dynamics} commute with the cotransition probabilities of $\G$ and hence extend to dynamics on the boundary, which are given by essentially the same formula after identifying the boundary with $\Sig_\infty$. Skew $Q$-polynomials generalize easily to infinite signatures: For $\nu,\lambda \in \Sig_\infty$, define
\begin{equation}
    Q_{\nu/\lambda}(\alpha) := \begin{cases}
    \alpha^{\sum_i \nu_i - \lambda_i} \varphi_{\nu/\lambda}& \text{$\nu_i \geq \lambda_i$ for all $i$ and $\sum_{i \geq 1} \nu_i - \lambda_i < \infty$} \\
    0 & \text{otherwise}
    \end{cases}
\end{equation}
where $\varphi_{\nu/\lambda}$ is extended from \Cref{def:psi_varphi_coefs} to infinite signatures in the obvious way. In the case $\nu,\lambda \in \Y$, this agrees with the standard branching rule in \cite{mac}. 

\begin{defi}
For $0 < \alpha < 1$, define 
\begin{equation}\label{eq:finite_markov}
    \Gamma_\alpha^n(\lambda,\nu) = Q_{\nu/\lambda}(\alpha) \frac{P_\nu(1,\ldots,t^{n-1})}{P_\lambda(1,\ldots,t^{n-1}) \Pi(\alpha;1,\ldots,t^{n-1})}
\end{equation}
for $n \in \Z_{\geq 1}$ and $\lambda,\nu \in \Sig_n$. For $\mu,\kappa \in \Y+D$, define
\begin{equation}\label{eq:infinite_markov_stable}
    \Gamma_\alpha^\infty(\mu, \kappa) = Q_{(\kappa-D[\infty])/(\mu-D[\infty])}(\alpha) \frac{P_{(\kappa-D[\infty])}(1,\ldots)}{P_{(\mu - D[\infty])}(1,\ldots) \Pi(\alpha;1,t,\ldots)}.
\end{equation}
Finally, for $\mu,\kappa \in \Sig_\infty^{unstable}$, define
\begin{equation}\label{eq:infinite_markov_unstable}
    \Gamma_\alpha^\infty(\mu, \kappa) = \lim_{D \to -\infty} \Gamma_\alpha^\infty(\mu^{(D)},\kappa^{(D)}).
\end{equation}

\end{defi}

When $\mu \in \Y$, the dynamics defined by \eqref{eq:infinite_markov_stable} yields a Hall-Littlewood process with one infinite specialization $1,t,\ldots$. The dynamics studied in \cite{van2021q} are a continuous-time limit of these: for positive real $\tau$, $(\Gamma^n_{\tau/D})^D$ converges to a Markov kernel, which defines a Markov process in continuous time $\tau$. We prove \Cref{thm:links_commutes} in the above discrete-time setting to minimize technicalities, though the statement for the limiting continuous-time process is the exactly analogous.

\begin{prop}\label{thm:links_commutes}
For $n \in \Z_{\geq 1} \cup \{\infty\}$, $\Gamma_\alpha^n$ is a Markov kernel. For $1 \leq n < m < \infty$ it commutes with the links $L^m_n$ in the sense that
\begin{equation}\label{eq:links_commute}
    \Gamma_\alpha^n L^m_n = L^m_n \Gamma_\alpha^m.
\end{equation}
Therefore given any coherent system $(M_n)_{n \geq 1}$ on $\G$, the pushforward measures $(M_n \Gamma_\alpha^n)_{n \geq 1}$ also form a coherent system. The induced map on $\partial \G$ is given by $\Gamma_\alpha^\infty$.
\end{prop}
\begin{proof}
The fact that \eqref{eq:finite_markov} and \eqref{eq:infinite_markov_stable} define Markov kernels follows directly from the Cauchy identity, \Cref{thm:finite_cauchy} and \eqref{eq:infinite_cauchy} respectively. For the infinite case \eqref{eq:infinite_markov_unstable}, we must show
\begin{equation}\label{eq:check_infinite_markov}
    \sum_{\kappa \in \Sig_\infty} \lim_{D \to -\infty} Q_{(\kappa^{(D)}-D[\infty])/(\mu^{(D)}-D[\infty])}(\alpha) \frac{P_{(\kappa^{(D)}-D[\infty])}(1,\ldots)}{P_{(\mu^{(D)} - D[\infty])}(1,\ldots) \Pi(\alpha;1,t,\ldots)} = 1.
\end{equation}
Note that
\[
Q_{(\kappa^{(D)}-D[\infty])/(\mu^{(D)}-D[\infty])}(\alpha) \frac{P_{(\kappa^{(D)}-D[\infty])}(1,\ldots)}{P_{(\mu^{(D)} - D[\infty])}(1,\ldots) \Pi(\alpha;1,t,\ldots)} \bbone(\kappa_i = \mu_i \text{ whenever }\kappa_i < D)
\]
increases monotonically as $D \to -\infty$ in a trivial way, namely it is either $0$ (for $D$ such that the indicator is $0$) or its final constant value (when the indicator function is nonzero). Hence we again interchange limit and sum by monotone convergence, obtaining 
\[
\lim_{D \to -\infty} \sum_{\substack{\kappa \in \Y+D}}  Q_{(\kappa-D[\infty])/(\mu^{(D)}-D[\infty])}(\alpha) \frac{P_{(\kappa-D[\infty])}(1,\ldots)}{P_{(\mu^{(D)} - D[\infty])}(1,\ldots) \Pi(\alpha;1,t,\ldots)}.
\]
This is $1$ by the Cauchy identity \eqref{eq:infinite_cauchy}.

Below we will show \eqref{eq:links_commute}, from which it follows that the maps $\Gamma_\alpha^n$ preserve coherent systems and hence induce a Markov kernel on $\partial \G$. To show that this Markov kernel is given by $\Gamma_\alpha^\infty$ we must show the `$m=\infty$' analogue of \eqref{eq:links_commute}, namely for any $\mu \in \Sig_\infty, \nu \in \Sig_n$ one has
\begin{equation}\label{eq:links_commute_infty}
    \sum_{\kappa \in \Sig_\infty} \Gamma_\alpha^\infty(\mu,\kappa) M^{\kappa}_n(\nu) = \sum_{\lambda \in \Sig_n} M^\mu_n(\lambda) \Gamma_\alpha^n(\lambda,\nu).
\end{equation}
We will treat \eqref{eq:links_commute} and \eqref{eq:links_commute_infty} simultaneously, and so introduce the notation $L^\infty_m(\mu,\cdot) := M^\mu_n(\cdot)$. For \eqref{eq:links_commute_infty}, if $\mu \in \Y+D$ for some $D$, then by translation-invariance and the Cauchy identity,
\begin{align*}
    \Gamma_\alpha^n L^\infty_n(\mu,\nu) &= \sum_{\lambda \in \Sig_n} L^\infty_n(\mu,\lambda ) \Gamma_\alpha^n(\lambda,\nu) \\
    &= \sum_{\lambda \in \Sig_n} L^\infty_n(\mu-D[\infty],\lambda -D[n]) \Gamma_\alpha^n(\lambda-D[n],\nu - D[n]) \\
    &= \sum_{\lambda \in \Sig_n} P_{(\mu - D[\infty])/(\lambda - D[n])}(t^n,\ldots) \frac{P_{(\lambda-D[n])}(1,\ldots,t^{n-1})}{P_{(\mu - D[\infty])}(1,\ldots)} \\
    &\times Q_{(\nu - D[n])/(\lambda - D[n])}(\alpha)  \frac{P_{(\nu-D[n])}(1,\ldots,t^{n-1})}{P_{(\lambda-D[n])}(1,\ldots,t^{n-1})\Pi(\alpha;1,\ldots,t^{n-1})} \\
    &= \frac{P_{(\nu-D[n])}(1,\ldots,t^{n-1})}{(1,\ldots,t^{n-1})\Pi(\alpha;1,\ldots,t^{n-1})}  \left(\frac{1}{\Pi(\alpha; t^n,\ldots)} \sum_{\kappa \in \Y} P_{\kappa/(\nu-D[n])}(t^n,\ldots)Q_{\kappa/(\mu - D[\infty])}(\alpha) \right)\\
    &= \sum_{\kappa \in \Y} L^\infty_n(\kappa+D[\infty],\nu) \Gamma_\alpha^\infty(\mu,\kappa+D[\infty]) \\
    &= L^\infty_n \Gamma_\alpha^\infty(\mu,\nu)
\end{align*}
The proof of \eqref{eq:links_commute} is the same after replacing $\infty$ with $m$, without the translation by $D$ issues. The case $\mu \in \Sig_\infty^{unstable}$ of \eqref{eq:links_commute_infty} requires a limiting argument: 
\begin{multline*}
    \Gamma_\alpha^n L^\infty_n(\mu,\nu) = \sum_{\lambda \in \Sig_n} Q_{\nu/\lambda}(\alpha) \frac{P_{\nu}(1,\ldots,t^{n-1})}{P_{\nu}(1,\ldots,t^{n-1}) \Pi(\alpha; 1,\ldots,t^{n-1})} \\
    \times \lim_{D \to -\infty} P_{(\mu^{(D)}-D[\infty])/(\lambda-D[n])}(t^n,\ldots) \frac{P_{(\lambda-D[n])}(1,\ldots,t^{n-1})}{P_{(\mu^{(D)}-D[\infty])}(1,\ldots)},
\end{multline*}
and by \Cref{thm:skew_formula} and monotone convergence this is equal to 
\[
\lim_{D \to -\infty} \sum_{\lambda \in \Sig_n}  \frac{Q_{\nu/\lambda}(\alpha)P_{\nu}(1,\ldots,t^{n-1})}{P_{\nu}(1,\ldots,t^{n-1}) \Pi(\alpha; 1,\ldots,t^{n-1})}P_{(\mu^{(D)}-D[\infty])/(\lambda-D[n])}(t^n,\ldots) \frac{P_{(\lambda-D[n])}(1,\ldots,t^{n-1})}{P_{(\mu^{(D)}-D[\infty])}(1,\ldots)}.
\]
Using that $\Gamma_\alpha^n(\lambda,\nu) = \Gamma_\alpha^n(\lambda-D[n],\nu-D[n])$ yields
\[
\lim_{D \to -\infty}\frac{P_{(\nu-D[n])}(1,\ldots,t^{n-1})}{P_{(\mu^{(D)}-D[\infty])}(1,\ldots)} \sum_{\lambda \in \Sig_n} Q_{(\nu-D[n])/(\lambda-D[n])}(\alpha) P_{(\mu^{(D)}-D[\infty])/(\lambda-D[n])}(t^n,\ldots).
\]
Applying the Cauchy identity \eqref{eq:infinite_cauchy} and the fact that 
\[
\Pi(\alpha;1,\ldots,t^{n-1})\Pi(\alpha;t^n,\ldots) = \Pi(\alpha;1,\ldots),
\]
and rearranging, yields
\[
\lim_{D \to -\infty} \sum_{\tkappa \in \Y} L_n^\infty(\tkappa,\nu-D[n]) \Gamma_\alpha^\infty(\mu^{(D)}-D[\infty],\tkappa).
\]
Changing variables to $\kappa = \tkappa + D[\infty]$ this is
\begin{equation}\label{eq:sumisequal}
\lim_{D \to -\infty} \sum_{\kappa \in \Y+D} L_n^\infty(\kappa-D[\infty],\nu-D[n]) \Gamma_\alpha^\infty(\mu^{(D)}-D[\infty],\kappa-D[\infty]).
\end{equation}
For each fixed $D$, there is an obvious bijection between $\Y+D$ and 
\[
\{\kappa \in \Sig_\infty^{unstable}: \kappa_i=\mu_i \text{ for all $i$ such that }\mu_i \leq D\},
\]
as signatures in either set are determined by their parts which are $>D$. Hence the sum in \eqref{eq:sumisequal} is equal to 
\begin{equation}\label{eq:change_to_unstable}
\sum_{\kappa \in \Sig_\infty^{unstable}} L_n^\infty(\kappa^{(D)}-D[\infty],\nu-D[n]) \Gamma_\alpha^\infty(\mu^{(D)}-D[\infty],\kappa^{(D)}-D[\infty]) I_D(\kappa,\mu),
\end{equation}
where 
\[
I_D(\kappa,\mu) := \bbone(\kappa_i=\mu_i \text{ for all $i$ such that }\mu_i \leq D)
\]
The summands in \eqref{eq:change_to_unstable}, as functions of $D$, take at most two values, namely $0$ (for all $\kappa \neq \mu$, for $D$ positive enough that the indicator function is $0$) and $L_n^\infty(\kappa,\nu) \Gamma_\alpha^\infty(\mu,\kappa)$ when the indicator function is nonzero. Hence monotone convergence again applies, yielding 
\[
\sum_{\kappa \in \Sig_\infty} \lim_{D \to -\infty} L_n^\infty(\kappa^{(D)}-D[\infty],\nu-D[n]) \Gamma_\alpha^\infty(\mu^{(D)}-D[\infty],\kappa^{(D)}-D[\infty]) I_D(\kappa,\mu).
\]
The summand stabilizes to $L_n^\infty(\kappa,\nu) \Gamma_\alpha^\infty(\mu,\kappa)$ (using translation-invariance of $L_n^\infty$), hence the above is equal to $L_n^\infty \Gamma_\alpha^\infty(\mu,\nu)$ as desired. This completes the proof.
\end{proof}




\newcommand{\etalchar}[1]{$^{#1}$}

\end{document}